\documentclass[ 11pt, leqno]{amsart}

\usepackage{amsmath, amsfonts, amssymb}

\newtheorem{thm}{Theorem}[section]
\newtheorem{lem}[thm]{Lemma}
\newtheorem{prop}[thm]{Proposition}
\newtheorem{cor}[thm]{Corollary}
\numberwithin{thm}{section}
\numberwithin{equation}{section}
\theoremstyle{definition}
\newtheorem{dfn}[thm]{Definition}

\def\n{\noindent}
\def\ov{\overline}

\def\hb{\hfil\break}
\def\n{\noindent}
\def\ov{\overline}

\usepackage{amsthm}

\usepackage{amsmath,amsthm,amscd,amssymb}
\usepackage{latexsym}
\usepackage{graphicx}
\usepackage[all]{xy}

\newcommand{\A}{\mathbb{A}}
\newcommand{\R}{\mathbb{R}}
\newcommand{\C}{\mathbb{C}}

\newcommand{\Q}{\mathbb{Q}}

\newcommand{\F}{\mathbb{F}}

\newcommand{\Z}{\mathbb{Z}}

\newcommand{\T}{\mathbb{T}}
\newcommand{\p}{\mathfrak{p}}
\newcommand{\q}{w} 

\newcommand{\h}{\mathfrak{t}}
\newcommand{\m}{\mathfrak{m}}

\newcommand{\aA}{\mathcal{A}}
\newcommand{\EE}{\mathcal{E}}

\newcommand{\FF}{\mathcal{F}}
\newcommand{\HH}{\mathcal{H}}
\newcommand{\OO}{\mathcal{O}}
\newcommand{\ZZ}{\mathcal{Z}}

\newcommand{\wt}{\widetilde}
\newcommand{\one}{\mathbf{1}}
\newcommand{\GG}{\mathbf{G}}
\newcommand{\ab}{\operatorname{ab}}
\newcommand{\ad}{\operatorname{ad}}
\newcommand{\Ad}{\operatorname{Ad}}
\newcommand{\qs}{\operatorname{qs}}
\newcommand{\sph}{\operatorname{sph}}
\newcommand{\old}{\operatorname{old}}
\newcommand{\new}{\operatorname{new}}
\newcommand{\diag}{\operatorname{diag}}
\newcommand{\antidiag}{\operatorname{antidiag}}

\newcommand{\Ind}{\operatorname{Ind}}

\newcommand{\Ann}{\operatorname{Ann}}

\newcommand{\rk}{\operatorname{rk}}

\newcommand{\End}{\operatorname{End}}
\newcommand{\Lie}{\operatorname{Lie}}

\newcommand{\Hom}{\operatorname{Hom}}
\newcommand{\into}{\hookrightarrow}
\newcommand{\onto}{\twoheadrightarrow}
\newcommand{\la}{\langle}
\newcommand{\ra}{\rangle}
\newcommand{\upi}{\pmb {\pi}}
\newcommand{\pr}{\operatorname{pr}}
\newcommand{\md}{\operatorname{mod}}
\newcommand{\der}{\operatorname{der}}

\newcommand{\im}{\operatorname{im}}
\newcommand{\Gal}{\operatorname{Gal}}

\newcommand{\GL}{\operatorname{GL}}
\newcommand{\GSp}{\operatorname{GSp}}

\newcommand{\GSpin}{\operatorname{GSpin}}
\newcommand{\Sp}{\operatorname{Sp}}

\newcommand{\SO}{\operatorname{SO}}

\newcommand{\PGL}{\operatorname{PGL}}
\newcommand{\PGSp}{\operatorname{PGSp}}

\newcommand{\Spec}{\operatorname{Spec}}

\newcommand{\SL}{\operatorname{SL}}

\newcommand{\St}{\operatorname{St}}
\newcommand{\U}{\operatorname{U}}

\newcommand{\rhoi}{\rho_\infty}
\newcommand{\rhoiL}{\rho_{\infty,L}}
\newcommand{\rhos}{\rho_\Sigma}
\newcommand{\rhosL}{\rho_{\Sigma,L}}
\newcommand{\Vs}{V_\Sigma}
\newcommand{\VsL}{V_{\Sigma,L}}
\newcommand{\Vi}{V_\infty}
\newcommand{\ViL}{V_{\infty,L}}
\newcommand{\bs}{\backslash}

\newcommand{\y}{\hspace{6pt}}

\begin{document}

\title{On Level-Raising Congruences}

\author {Yuval Z. Flicker}

\thanks{\n Keywords: Level-Raising, Congruences, Algebraic Modular Forms.\hb
}


\address{\noindent Department of Mathematics, The Ohio State University,
231 W. 18th Ave., Columbus, OH 43210-1174; email: flicker@math.osu.edu}

\date{}

\begin{abstract}
In this paper we rewrite a work of Sorensen to include nontrivial 
types at the infinite places. This work extends results of K.  Ribet 
and R. Taylor on level-raising for algebraic modular forms on
$D^{\times}$, where $D$ is a definite quaternion algebra over a
totally real field $F$. We do this for any automorphic representations
$\pi$ of an arbitrary reductive group $G$ over $F$ which is
compact at infinity. We do not assume $\pi_\infty$ is trivial. 
If $\lambda$ is a finite place of $\bar{\Q}$,
and $w$ is a place where $\pi_w$ is unramified and $\pi_w \equiv
\one$ (mod $\lambda$), then under some mild additional
assumptions (we relax requirements on the relation between $w$ and
$\ell$ which appear in previous works) we prove the existence of a 
$\tilde{\pi} \equiv \pi$ (mod $\lambda$) such that $\tilde{\pi}_w$ 
has more parahoric fixed vectors than $\pi_w$. In the case where 
$G_w$ has semisimple rank one, we sharpen results of Clozel, Bellaiche 
and Graftieaux according to which $\tilde{\pi}_w$ is Steinberg. To 
provide applications of the main theorem we consider two examples over
$F$ of rank greater than one. In the first example we take $G$ to be a 
unitary group in three variables and a split place $w$. In the second 
we take $G$ to be an inner form of GSp(2). In both cases, we obtain
precise satisfiable conditions on a split prime $w$ guaranteeing
the existence of a $\tilde{\pi} \equiv \pi$ (mod $\lambda$) such
that the component $\tilde{\pi}_w$ is generic and Iwahori
spherical. For symplectic $G$, to conclude that $\tilde{\pi}_w$ is
generic, we use computations of R. Schmidt. In particular, if
$\pi$ is of Saito-Kurokawa type, it is congruent to a
$\tilde{\pi}$ which is not of Saito-Kurokawa type.
\end{abstract}

\maketitle 

\section*{Introduction}

This paper stems from the following result of Ribet [R]:

\begin{thm}
Let $f \in S_2(\Gamma_0(N))$ be an eigenform. Let
$\lambda|\ell$ be a finite place of $\bar{\Q}$ with $\ell
\geq 5$ and $f$ not congruent to an Eisenstein series modulo
$\lambda$. Let $q$ be a prime number with $(q,N\ell)=1$ and 
$a_q(f)^2 \equiv (1+q)^2\,\,\, (\md\lambda).$
Then there exists a $q$-new eigenform $\tilde{f} \in
S_2(\Gamma_0(Nq))$ congruent to $f$ mod $\lambda$.
\end{thm}

Two eigenforms $f$ and $\tilde{f}$ are said to be {\it congruent modulo}
$\lambda$ if their Hecke eigenvalues are algebraic integers congruent 
for almost all primes, that is, if $a_p(f)\equiv a_p(\tilde{f})$ (mod $\lambda$)
for almost all $p$. The proof of this theorem can be reduced, via
the correspondence from an inner form to GL(2) (see [F] for a simple 
proof),  to the corresponding statement for $D^{\times}$ where $D$ 
is a definite quaternion algebra over $\Q$.

A goal of this paper is to prove that an automorphic form of
Saito-Kurokawa type is congruent to an automorphic form which is
$not$ of Saito-Kurokawa type. By functoriality ([F1]) the statement
can be reduced to that for an inner form $G$ of $\PGSp(2)/F$
such that $G(\R)$ is compact. Indeed, the set of packets of
automorphic representations of $G(\A)$ can be identified with
a subset of the set of such objects on $\PGSp(2,\A)$, where
almost all local components are the same.
By a form on $G\simeq \SO(5)$ of
Saito-Kurokawa type we mean the lift of $\one\times\rho$
from the endoscopic group $\PGL(2,\A)\times\PGL(2,\A)$ to 
$\PGSp(2,\A)$, where $\rho$ is cuspidal and $\one$ is
trivial on $\PGL(2,\A)$. It is nontempered at almost all places.
We achieve this goal in Theorem 0.6, proven in Section 12.

We apply ideas and methods of R. Taylor [T1] and [T2]. The
level-raising part of Taylor's proof carries over to the following more
general setup. Let $F$ denote a totally real number field with ring
$\A=F_{\infty}\times \A^{\infty}$ of ad\`eles. We denote the set of real 
places of $F$ by $\infty$. Let $G$ be a connected reductive 
$F$-group such that $G_{\infty}^1:=G_{\infty}\cap G(\A)^1$ (see Sect. 3)
is compact and the derived group $G^{\der}$ is simple and simply 
connected; here $G_\infty=G(F_\infty)$. When $F \neq \Q$, this just 
means that $G_{\infty}$ is compact, see Prop. 3.1 below. However, 
when $F=\Q$ and the $\Q$- and $\R$-ranks of $G^{\ab}$ are equal, 
it suffices that $G_{\infty}^{\der}$ be compact. Here
$G^{\ab}$ denotes the biggest quotient group of $G$ which is a torus, 
thus $G^{\ab}=G/G^{\der}$. There are plenty of such groups $G$. In fact, 
any split simple $F$-group not of type $A_n$ ($n \geq 2$), $D_{2n+1}$ 
or $E_6$ has infinitely many inner forms which are compact at
infinity (and quasi-split at all but at most one finite place).

Fix a Haar measure $\mu=\otimes \mu_v$ on
$G(\A^{\infty})$. We state the results using the
following notion of congruence. As $K$ varies over the compact
open subgroups of $G(\A^{\infty})$, the centers
$Z(\HH_{K,\Z})$ of the Hecke algebras form an inverse
system. To an automorphic representation $\pi$ of $G(\A)$ we
associate the character
$\eta_{\pi}:\underleftarrow \lim Z(\HH_{K,\Z}) \to\C$
such that $\eta_{\pi}=\eta_{\pi^K} \circ \pr_K$ for every
compact open subgroup $K$ such that $\pi^K \neq 0$. The character
$\eta_\pi$ takes its values on $Z(\HH_{K,\Z})$ in the ring of
integers of some number field, depending on $F$, $G$ and $K$.
If $\lambda$ is a finite place of $\bar{\Q}$, we say that $\tilde{\pi}$ 
and $\pi$ are {\it congruent modulo} $\lambda$ if their characters are. 
Write $\tilde{\pi}\equiv \pi$ (mod $\lambda$) in this case. A similar 
notion makes sense locally, and then $\tilde{\pi}\equiv \pi$ (mod
$\lambda$) if and only if $\tilde{\pi}_v\equiv \pi_v$ (mod
$\lambda$)  for all finite $v$. Moreover, when both $\tilde{\pi}_v$ 
and $\pi_v$ are unramified, $\tilde{\pi}_v\equiv\pi_v$ (mod $\lambda$) 
simply means that the Satake parameters are
congruent. 

\begin{dfn}
Let $\pi$ be an automorphic representation of $G(\A)$.
Let $\lambda$ be a finite place of $\bar{\Q}$.
We say that $\pi$ is {\it abelian modulo} $\lambda$ {\it 
relative to} $K$ if $\pi^K \neq 0$ and there exists an
automorphic character $\chi$ of $G(\A)$, trivial on $K$,
such that $\eta_{\pi^K}(\phi) \equiv \eta_{\chi}(\phi)$
$(\md\lambda)$, $\forall \phi \in Z(\HH_{K,\Z})$.
We say that $\pi$ is {\it abelian modulo} $\lambda$ if it is
abelian modulo $\lambda$ relative to some $K$, thus
$\pi\equiv \chi$ (mod $\lambda$) for some $\chi$.
\end{dfn}

This is the analogue of the notion Eisenstein modulo $\lambda$ 
in [Cl, p. 1269]. Since $G^{\der}$ is anisotropic in our
setup, there are neither cusps nor Eisenstein series. Thus 
the terminology abelian modulo $\lambda$ seems more suitable. 

The following theorem is in some sense the main result of this paper.

Let $F$ be a totally real number field, $\Sigma$ a finite set of 
finite places of $F$.  
Fix a compact open subgroup $K_v$ of $G_v=G(F_v)$ for each 
$v\notin\infty$, hyperspecial for almost all $v$. Fix an irreducible
representation $\rho_\Sigma$ of $K_\Sigma=\prod_{v\in\Sigma}K_v$
and an irreducible smooth unitary representation $\rho_\infty$ of 
$G_\infty=\prod_{v\in\infty}G_v$. Then $K=\prod_{v \notin \Sigma}K_v$ 
is a compact open subgroup of $G(\A^{\Sigma})$. Denote by $e_K$
the constant measure supported on $K$ of volume 1. It is the unit
element in the Hecke algebra $\HH_{K,\Z}$. 

\begin{thm}
Let $\lambda|\ell$ be a finite place of $\bar{\Q}$ such that there 
exists at least two finite places $v$ where $\ell\nmid |K_v|$ $($this 
is automatic if there is an $F$-embedding $G\hookrightarrow\GL(n)$
and $\ell > [F:\Q]n+1)$. Let $\pi=\otimes \pi_v$ be an automorphic 
representation of $G(\A)$ such that $\pi_{\infty}=\rho_\infty$, 
$\pi_\Sigma\supset\rho_\Sigma$, and $\pi^K \neq 0$. Assume $\pi$ 
is nonabelian modulo $\lambda$ relative to $K$. Let $w$ be a finite 
place of $F$ such that $K_w$ is hyperspecial. Let $q=q_w$ denote the
residual cardinality of $w$. Let $J_w=K_w \cap K_w'$ 
be a parahoric subgroup, where $K_w' \neq K_w$ is maximal 
compact. Let $J=J_wK^w$ and $K'=K_w'K^w$. Put
$$[K'_w:J_w]_{K_w}=[K'_w:J_w]/([K'_w:J_w],[K_w:J_w])$$
and
$$e_{K,K'}=[K_w:J_w][K'_w:J_w]_{K_w}(e_K \ast e_{K'}\ast e_K) 
\in Z(\HH_{K,\Z}).$$
Assume $\ell\nmid q_w[K'_w:J_w]_{K_w}$ and $(\star)$
$\eta_{\pi^{K}}(e_{K,K'}) \equiv \eta_{\one}(e_{K,K'})\,(\md\lambda).$
Then there exists an automorphic representation
$\wt{\pi}=\otimes \wt{\pi}_v$ of $G(\A)$ such that
$\wt{\pi}_{\infty}=\rho_\infty$, $\wt{\pi}_\Sigma\supset\rho_\Sigma$ 
and $\wt{\pi}^{K^w} \neq 0$ satisfying  $\wt{\pi}_w^{J_w} \neq 
\wt{\pi}_w^{K_w}+\wt{\pi}_w^{K_w'}$, and 
$\eta_{\wt{\pi}^J}(\phi) \equiv \eta_{\pi^K}(e_K \ast\phi)$ 
$(\md\lambda)$ for all $\phi \in Z(\HH_{J,\Z})$.
\end{thm}

This theorem claims nothing unless $\pi_w^{J_w}=\pi_w^{K_w}+
\pi_w^{K_w'}$.  The assumption $(\star)$ is implied by the stronger
assumption: ($\star\star$) $\pi_w$ is congruent to the trivial representation
$\one$ modulo $\lambda$, namely $\pi(\phi)\equiv\one(\phi)$
$(\md\lambda)$, $\forall \phi \in Z(\HH_{K,\Z})$.
But $(\star)$ is strictly weaker
than ($\star\star$). Our final conclusion is slightly more precise than 
$\wt\pi\equiv\pi$ $(\md\lambda)$. If $G_w^{\der}$ has rank one, $J_w$ 
is an Iwahori subgroup $I_w$ and we show in Lemma 11.1 that 
$\tilde{\pi}_w^{K_w}=0$ (and $\tilde{\pi}_w^{K'_w}=0$) but 
$\tilde{\pi}_w^{I_w}\neq 0$. The eigensystem of a modular form 
mod $\ell$ comes from an algebraic modular form mod 
$\ell$ on $D^{\times}$, where $D/\Q$ is the quaternion algebra with 
ramification locus $\{\infty,\ell\}$, see Serre [S]. Using the transfer of 
automorphic forms from $D(\A)^\times$ to the split form $\GL(2,\A)$ 
(see [F] for a simple proof) we get the result of Ribet after stripping 
powers of $\ell$ from the level. Note that $[K'_w:J_w]_{K_w}=1$ when $K'_w$
is conjugate to $K_w$. The condition $\ell\nmid [K'_w:J_w]$ which appears
in [So] introduces the requirement $(\ell,1+q)=1$ in the formulation of
Ribet's theorem in [So].

There is another proof of Ribet's theorem relying on the so-called
Ihara lemma. It states that for $q \nmid N\ell$, the degeneracy
maps $X_0(Nq) \rightrightarrows X_0(N)$ induce an injection
$H^1(X_0(N),\Z_{\ell})^{\oplus 2} \to H^1(X_0(Nq),\Z_{\ell})$
with torsion-free cokernel. The proof of this lemma reduces to the
congruence subgroup property of the group $\SL(2,\Z[1/q])$. In our
case we are looking at functions on a finite set, and the analogue
of the Ihara lemma can be proved by imitating the combinatorial
argument of Taylor [T1, p. 274] in the diagonal weight $2$ case.
See section 7.3 below.

Here are a few applications of Theorem 0.3. 

\begin{thm}
Let $F$ be a totally real number field. Let $\pi$ be as in Theorem $0.3$. 
Let $w$ be a finite place of $F$ such that $K_w$ is hyperspecial and the 
$F_w$-rank of $G_w^{\der}$ is one. Let $I_w=K_w \cap K_w'$ be an Iwahori 
subgroup, where $K_w'\neq K_w$ is maximal compact. Put $I=I_wK^w$ and 
$K'=K_w'K^w$. Suppose $\ell\nmid q_w[K'_w:I_w]_{K_w}$
and $\eta_{\pi^{K}}(e_{K,K'}) \equiv \eta_{\one}(e_{K,K'})(\md\lambda),$
with $e_{K,K'}$ as in Theorem $0.3$. Then there exists an automorphic
representation $\wt{\pi}=\otimes \wt{\pi}_v$ of $G(\A)$ such that 
$\wt{\pi}_{\infty}=\rho_\infty$, $\wt{\pi}^{K^w} \neq 0$, 
$\wt{\pi}_w^{I_w}\neq 0$, $\wt{\pi}_w^{K_w}=0=\wt{\pi}_w^{K'_w}$,
satisfying $\eta_{\wt{\pi}^I}(\phi)\equiv\eta_{\pi^K}(e_K\ast\phi)
(\md\lambda)$ for all $\phi \in Z(\HH_{I,\Z})$. If 
$\pi_\Sigma\supset\rho_\Sigma$ then $\wt{\pi}$ can be chosen 
to satisfy $\wt{\pi}_\Sigma\supset\rho_\Sigma$.
\end{thm}

This theorem is a variant of Bellaiche's Theorem 1.4.6, [Bel, p. 215]: It 
gives results modulo arbitrary $\lambda|\ell$ prime to $q_w[K'_w:I_w]_{K_w}$, 
independently of $\pi$, the level-raising condition is weaker, and we get 
information about the action of the center of the Iwahori-Hecke algebra on 
$\wt{\pi}_w^{I_w}$. Bellaiche's proof is different. He uses results 
of Lazarus and Vigneras from modular representation theory, such 
as the computation of the composition series of universal modules. 
His $\ell$ is prime to $q_w$ times the number of neighbors of the 
vertex in the Bruhat-Tits building fixed by $K_w$, times the number of 
neighbors of the vertex fixed by $K'_w$, and has to lie outside a 
finite set depending on $\pi$, but his $\pi$ is not required to be
nonabelian mod $\lambda$. His level-raising condition: 
$\eta_{\pi^K}(\phi) \equiv \eta_\one(\phi)$ for $all$ $\phi\in\HH_{K_w}$, 
is stronger, and he can conclude that $\wt{\pi}_w$ is the actual 
Steinberg representation of $G_w$. We show this too, using the analysis
of section 11. In [So] it is only shown that $\wt{\pi}_w$ is ramified.
See also [BG] where general $\rho_\infty$ are considered, and the only 
condition on $\ell\nmid q_w$ is that it lies outside an unknown finite 
set depending on $\pi$, but $\pi$ is not assumed to be nonabelian mod 
$\lambda$.

Consider the special case where $E/F$ is a totally imaginary quadratic 
extension of a totally real number field $F$, $G^{\qs}=\U(2,1)$ is 
the quasi-split unitary $F$-group in $3$ variables split over $E$, and 
$G=\U(3)$ is an inner form of $G^{\qs}$ such that $G_{\infty}$ is compact. 
For $F$-primes $\q$ inert in $E$, the semisimple rank of $G(F_{\q})$ is 
one. In this case Theorem 0.4 strengthens (to $\ell\nmid q_w$) Clozel [Cl] 
(where $F=\Q$ and $\ell\nmid q_w(q_w^3+1)(q_w-1)$), [Bel] Theorem 1.4.6, 
where $\ell\nmid q_w(q_w^3+1)$ and -- as in [BG] -- $\ell$ 
is outside a finite set depending on $\pi$. Indeed, $[K'_w:I_w]=q_w+1$ 
divides $[K_w:I_w]=q_w^3+1$, hence $[K'_w:I_w]_{K_w}=1$, so our condition 
on $\ell$ is only that it be prime to $q_w$. From $\wt{\pi}_w^{I_w}\neq 0$ 
and $\wt{\pi}_w^{K_w}=0=\wt{\pi}_w^{K'_w}$ we conclude that $\wt{\pi}_w$ 
is Steinberg (as $\pi^\times{}^{K}\not=0$ and $\pi^+{}^{K'}\not=0$).
We recall the classification of reducible unramified induced representations,
in particular in the case of $G=\U(3)$, in section 11.3.

If $\pi$ is a representation of $G(\A)=\U(3,\A)$ such that $\pi_v$ is 
the nontempered $\pi_v^\times$ for 
almost all $v$ (in the notations of [F2]), and $\pi_w=\pi_w^\times$, 
then $\wt{\pi}_w$ is not the cuspidal $\pi_w^-$ (since 
$\wt{\pi}_w^{I_w}\not=0$) and not $\pi_w^\times$
(as $\wt{\pi}_w^{K_w}=0$), so $\wt{\pi}$ has no component 
$\pi_v^\times$, by the results of [F2]. Alternatively, 
this follows from $\wt{\pi}_w$ being Steinberg.

When $G=\U(3)$ and the $F$-prime $w$ splits in $E$, thus 
$G(F_w)=\GL(3,F_w)$, we obtain the following as a corollary.

\begin{thm}
Let $\pi=\otimes \pi_v$ be an automorphic representation of $G(\A)$ 
with $\pi_{\infty}=\rho_\infty$. Choose a compact open subgroup 
$K=\prod K_v \subset G(\A^{\infty})$ such that $\pi^K \neq 0$. 
Let $\lambda|\ell$ be a finite place of $\bar{\Q}$ such that $\pi$ is 
nonabelian modulo $\lambda$ relative to $K$. If $\ell \leq 3$, or 
if $E=\Q(\sqrt{-7})$ and $\ell=7$, assume
$\ell \nmid |K_v|$ for at least two primes $v$. Let $\q \nmid \ell$
be a prime in $F$, split in $E$, such that $K_\q$ is hyperspecial. 
Suppose  
there is a $\q_E|\q$ such that the Satake parameter is
$\h_{\pi_{\q_E}}\equiv\diag(q_\q,1,q_\q^{-1})(\md\lambda).$
Then there exists an automorphic representation
$\wt{\pi}=\otimes \wt{\pi}_v$ of $G(\A)$ with
$\wt{\pi}_{\infty}=\rho_\infty$ and $\wt{\pi}^{K^\q}\neq 0$
satisfying $(1)$ $\wt{\pi}_\q$ is either an irreducible unramified
principal series or induced from a Steinberg representation $($in
particular $\wt{\pi}_\q$ is generic, not square integrable$)$, 
and $\wt{\pi}_\q^{J_\q}\neq 0$ for any maximal proper parahoric 
subgroup $J_\q$, and $(2)$ $\eta_{\wt{\pi}^J}(\phi)\equiv \eta_{\pi^K}
(e_K \ast\phi)(\md\lambda)$ for all $\phi \in Z(\HH_{J,\Z})$,
where $J=J_\q K^\q$, hence $\tilde{\pi}\equiv \pi\, (\md\lambda)$.
\end{thm}

Theorem 0.5 claims nothing unless $\pi_\q$ is induced
from the determinant (type IIb of Tables A, B in Section 11), 
that is, unramified and non-generic (and not $1$-dimensional), 
which is the case for the endoscopic lifts from $\U(2) \times \U(1)$ 
considered in [Bel, p. 250].  Since we deal with any $\pi_{\infty}= 
\rho_\infty$, it follows that if $\pi$ is endoscopic abelian (that 
is, a lift of a character of a proper endoscopic group), 
then it is congruent to a $\tilde{\pi}$ which is not endoscopic 
abelian. This is true even for $\U(n)$, for all $n \geq 2$. For 
$n=3$ this result has been applied to the Bloch-Kato conjecture 
for certain Hecke characters of $E$ [Bel]. In fact, the results 
one can get for $\U(n)$ indicate that an endoscopic abelian lift 
$\pi$ is congruent to a $\wt{\pi}$ which is not endoscopic abelian. 
We cannot prove by our methods that $\tilde{\pi}_\q$ is ramified.
In his thesis [Bel, p. 218], Bellaiche also has a result in the 
split case. His $\ell$ is prime to $q_w(q_w^3-1)(q_w+1)$, and 
lies outside a finite set depending on $\pi$. If $\pi$ occurs with 
multiplicity one (the multiplicity one theorem for U(3) is currently 
proven -- in [F2] -- only for representations satisfying some mild 
condition at the dyadic places; it is not yet proven in general, 
contrary to the assertion of [Cl]), he obtains a $\wt{\pi}$ with
$\wt{\pi}_\q$ ramified. We classify the Iwahori-spherical 
representations of GL(3) and compute the dimensions of their 
parahoric fixed spaces. This allows us to conclude that 
$\tilde{\pi}_\q$ is either a full unramified principal series or 
induced from a Steinberg representation. Hence, from our analysis,
$\wt{\pi}_\q$ is induced from Steinberg. Theorem 0.5 seems related 
to the $n=3$ case of conjecture 5.3 in [T2, p. 35], providing an analogue 
of Ihara's lemma. Automorphic representations of unitary groups with 
a generic component at a split prime come up naturally in the proof of 
the local Langlands correspondence for $\GL(n)$ [HT].

Next, let $G$ be an inner form of GSp(2) such that $G^{\der}(\R)$ 
is compact. Concretely, $G=\GSpin(f)$ for some definite quadratic 
form $f$ in $5$ variables over a totally real $F$. In this situation, 
Theorem 0.3 yields:

\begin{thm}
Let $\pi=\otimes \pi_v$ be an automorphic representation of
$G(\A)$ with $\pi_{\infty}=\rho_\infty$. Choose a compact open 
subgroup $K=\prod K_v$ such that $\pi^K \neq 0$. Let $\lambda|\ell$ 
be a finite place of $\bar{\Q}$ such that $\pi$ is nonabelian modulo
$\lambda$ relative to $K$. If $\ell \leq 5$ assume $\ell \nmid |K_v|$ for at
least two primes $v$. Let $\q \nmid \ell$ be a prime such that $K_\q$
is hyperspecial. Suppose $\h_{\pi_\q\otimes \nu_\q^{-3/2}}\equiv
\diag(1,q_w,q_w^2,q_w^3)(\md\lambda).$
Then there exists an automorphic representation
$\wt{\pi}=\otimes \wt{\pi}_v$ of $G(\A)$ with
$\wt{\pi}_{\infty}=\rho_\infty$ and $\wt{\pi}^{K^\q}\neq 0$
satisfying $(1)$ $\wt{\pi}_\q$ is generic and Heisenberg-spherical, 
and $(2)$ $\eta_{\wt{\pi}^J}(\phi)\equiv \eta_{\pi^K}(e_K \ast
\phi)(\md\lambda)$ for all $\phi \in Z(\HH_{J,\Z})$,
where $J=J_\q K^\q$.

If in addition $q^4 \neq 1(\md\ell)$, then $\wt{\pi}_\q$ must be of type 
I, IIa or IIIa of Tables C, D in Appendix $2$.
\end{thm}

By the Heisenberg parahoric we mean the inverse image, under the 
reduction map, of the standard maximal parabolic in $\GSp(2,\F_\q)$
whose unipotent radical is a Heisenberg group. The proof relies on 
computations of R. Schmidt [Sch]. If $m(\pi)=1$, Bellaiche's methods 
seem to apply and one can probably show that $\tilde{\pi}_\q$ is 
induced from a twisted Steinberg representation on the standard 
Heisenberg-Levi. It is known (see, e.g., [F1]) that Saito-Kurokawa lifts 
(that is, lifts of $\one\times$cuspidal from $\PGL(2,\A)\times\PGL(2,\A)$
to $\PGSp(2,\A)$) are locally non-generic everywhere. Therefore, if $\pi$
is of Saito-Kurokawa type, it is congruent to a $\tilde{\pi}$ which is not 
of Saito-Kurokawa type. The interest in it stems from hoped for applications
to the Bloch-Kato conjecture for the motives attached to classical modular 
forms. In particular, one hopes to establish a mod $\ell$ analogue of a 
result of Skinner and Urban [SU], which is valid for $all$ (not necessarily 
ordinary) modular forms of classical weight at least 4.

There exists $q$ with $q^4 \neq 1$ (mod $\ell$)
precisely when $\ell \geq 7$. In this case $\wt{\pi}_\q$ is an
unramified principal series (type I) or induced from a twisted
Steinberg representation $\chi \St_{\GL(2)} \rtimes \chi'$
or $\chi \rtimes \chi' \St_{\GL(2)}$ (type IIa and IIIa
respectively). If one can show that $\wt{\pi}_\q$ is
para-ramified, meaning that $\wt{\pi}_\q$ has no nonzero
$K_\q'$-fixed vectors, one can conclude that it is of type IIIa and
therefore induced from a twisted Steinberg representation on the
Heisenberg-Levi, since $m(\pi)=1$ (see [F1]), using the methods of 
[Bel] and [Cl]. The result above only gives nontrivial congruences
if $\pi_\q$ is nongeneric. If $\pi$ is of Saito-Kurokawa type, it is
locally nongeneric, and we get a $\wt{\pi}$ congruent to $\pi$
which is not of Saito-Kurokawa type. If we know that $\wt{\pi}_\q$ 
is of type IIIa, we can apply this strategy to the Bloch-Kato
conjecture for the motives attached to classical modular forms of
weight (at least) $4$, using the methods of [Bel]. We should note
that if we choose to work with the Siegel-parahoric $J_\q'$, we can
only conclude that $\wt{\pi}_\q$ is generic $or$ a
Saito-Kurokawa lift.

This work is simply an attempt to complete the beautiful paper [So] by
extending it from the special case $\pi_\infty=\one$ to permit $\pi_\infty$ 
to be any irreducible continuous representation $\rho_\infty$ of $G_\infty$. 
Further we optimize the constraint on $\ell$ and determine $\wt{\pi}_w$
to be Steinberg in the case of U(3). Except for these minor changes, we 
follow [So] very closely, attempting to expand some of the arguments there.

\section{The Abstract Setup}

In this section, we fix a ring $\OO$ of characteristic zero which is a
finite product of domains. Denote by $L$ the associated product of 
fields of fractions. There are two cases of interest for us. The first is 
where $\OO$ is the ring of integers in a number field $L\subset\C$.
The case that we shall actually use in this paper is as follows.
Let $L_1\subset\C$ be a number field such that $[L_1:L_0]=2$
where $L_0=L_1\cap\R$. Let $\lambda$ be a finite place of $L_1$,
and $\lambda_0=\lambda\cap L_0$ the place of $L_0$ under $\lambda$. 
Let $L_{1 \lambda}$ be the completion of $L_1$ at $\lambda$, and 
$(L_0)_{\lambda_0}$ the completion of $L_0$ at $\lambda_0$. Then 
$L_{1 \lambda_0}=L_1\otimes_{L_0}(L_0)_{\lambda_0}$ is 
$L_{1 \lambda}$ if $\lambda_0$ stays prime in $L_1$, and it is 
$L_{1 \lambda}\oplus L_{1 \ov\lambda}$ if $\lambda_0$ splits as 
$\lambda\ov\lambda$ in $L_1$. Note that 
$L_{1 \ov\lambda}\simeq L_{1 \lambda}$. The ring of integers 
$R_{\lambda_0}$ in $L_{1 \lambda_0}$ is the ring of integers 
$R_\lambda$ in $L_{1 \lambda}$ if $\lambda_0$ stays prime, 
and $R_\lambda\oplus R_{\ov\lambda}$ if $\lambda_0$
splits. Then the case we shall actually use is that where 
$\OO=R_{\lambda_0}$ and $L=L_{1 \lambda_0}$.

Let $H$ be a commutative $L$-algebra. We do not require $H$ 
to be of finite dimension. However, we assume $H$ comes equipped 
with an involution $\phi\mapsto \phi^{\vee}$. An {\it involution} is 
an $L$-linear anti-automorphism (thus 
$(\phi_1\phi_2)^\vee=\phi_2^\vee\phi_1^\vee$)
of order two. Moreover, we fix an $\OO$-order $H_{\OO}\subset H$ 
preserved by $\vee$. An $\OO$-{\it order} is an 
$\OO$-subalgebra which is the $\OO$-span of an $L$-basis
for $H$. Then we look at a triple
$(V, \la -,- \ra_V, V_{\OO})$ consisting of:\hb
(1) $V$ is a finite-dimensional $L$-space with an action
$r_V:H \to \End_{L}(V)$;\hb
(2) $\la -,- \ra_V$ is a nondegenerate, symmetric,
$L$-bilinear form $V \times V \to L$;\hb
(3) $V_{\OO}\subset V$ is an $\OO$-{\it lattice} (that is, the
$\OO$-span of an $L$-basis).

We impose the following compatibility conditions on these data:\hb
(1) $r_V(\phi^{\vee})$ is the adjoint of $r_V(\phi)$ 
with respect to $\la -,- \ra_V$;\hb
(2) $V_{\OO}\subset V$ is preserved by the order 
$H_{\OO}\subset H$;\hb
(3) $V_{\OO}/(V_{\OO}\cap V_{\OO}^{\vee})$ and
$V_{\OO}^{\vee}/(V_{\OO}\cap V_{\OO}^{\vee})$ are torsion
$\OO$-modules.

Here $V_{\OO}^{\vee}=\{v \in V: \la v,V_{\OO}\ra_V \subset \OO \}$
is the {\it dual lattice} of $V_{\OO}$ in $V$. 

Choose nonzero annihilators $A_V $ and $B_V$ in $\OO$ of 
the torsion modules above, that is, such that $A_V V_{\OO}\subset 
V_{\OO}^{\vee}$ and  $B_V V_{\OO}^{\vee}\subset V_{\OO}$, thus
$$
\text{$A_V \la V_{\OO},V_{\OO}\ra_V \subset \OO$ $\y$ and $\y$
$\la v,V_{\OO}\ra_V \subset \OO \Rightarrow B_Vv \in V_{\OO}$.}
$$

Let $(U, \la -,- \ra_U, U_{\OO})$ be another such triple.
Choose annihilators $A_U$ and $B_U$ for it too. Suppose we
are given an $H$-linear map $\delta: U \to V$ satisfying:\hb
(1) $U=\ker\delta\oplus(\ker\delta)^{\bot}$;  
$V=\im\delta\oplus(\im\delta)^{\bot}$;\hb
(2) $\delta(U_{\OO}) \subset V_{\OO} \cap \delta(U)$, and the
quotient is killed by $C \in \OO-\{0\}$.

Put $V^{\old}=\im\delta$ and $V^{\new}=(\im\delta)^{\bot}$. 
These are $H$-stable subspaces of $V$. By assumption we have 
an orthogonal decomposition $V=V^{\old} \oplus V^{\new}$.
The adjoint map $\delta^{\vee}:V \to U$ is
defined by $\la u,\delta^{\vee}v\ra_U=\la\delta u,v\ra_V$.
Then $\delta^{\vee}$ maps $V^{\new}$ to 0, and 
$\delta^{\vee}:V^{\old}\to (\ker\delta)^{\bot}$ is injective, 
with inverse $\delta$.

\begin{dfn}
Put $V_{\OO}^{\old}=V_{\OO}\cap V^{\old}$ and
$V_{\OO}^{\new}=V_{\OO}\cap V^{\new}$.
\end{dfn}

These $H_{\OO}$-stable submodules of $V_{\OO}$
span $V^{\old}$ and $V^{\new}$ respectively. They are
orthogonal, but their sum is not always all of $V_{\OO}$.
Note that $\delta(U_{\OO}) \subset
V_{\OO}^{\old}$ and $C V_{\OO}^{\old}
\subset \delta(U_{\OO})$ by assumption. 

\begin{dfn}
Let $\T_{\OO}$ be the image of $H_{\OO}$ in 
$\End_{\OO}(V_{\OO})$.  Let
$$
\text{$\T_{\OO}^{\old} \subset
\End_{\OO}(V_{\OO}^{\old})$ $\y$ and $\y$
$\T_{\OO}^{\new} \subset
\End_{\OO}(V_{\OO}^{\new})$}
$$
denote the images of $H_{\OO}$ defined by these
submodules. 
\end{dfn}

Clearly we have natural surjective maps $\T_{\OO} \onto
\T_{\OO}^{\old}$ and $\T_{\OO}\onto \T_{\OO}^{\new}$ 
given by restriction, and $\T_{\OO}$ acts faithfully on $V_{\OO}$.

\section{Taylor's Lemma}

By a {\it congruence module} we mean a
$\T_{\OO}$-module, such that the action factors through
both $\T_{\OO}^{\old}$ and
$\T_{\OO}^{\new}$. The following lemma was stated
for $\OO=\Z$, trivial annihilators, and injective $\delta$
in [T2, p. 331]

\begin{lem}
Put $E=A_UB_VC^2$ and $U_{\OO}'=U_{\OO}\cap (\ker\delta)^{\bot}$. 
Then $U_{\OO}''=U_{\OO}'/(U_{\OO}' \cap
E^{-1}\delta^{\vee}\delta(U_{\OO}))$ is a congruence module.
\end{lem}

\begin{proof}
Since $U_{\OO}$ is preserved by $H_{\OO}$ and $\delta$ is $H$-linear,
the algebra $\T_{\OO}$ acts naturally on $U_{\OO}'$  via the embedding 
of $U_{\OO}'$ into $V_{\OO}^{\old}$ defined by $\delta$. This action factors 
through $\T_{\OO}^{\old}$. 

It remains to show that $\T_{\OO}^{\new}$ acts on 
$U_{\OO}''$, namely that the action is well defined. So
suppose that $\phi \in H_{\OO}$ acts as zero on
$V_{\OO}^{\new}$. We must show that $E\phi$ maps
$U_{\OO}'$ into $\delta^{\vee}\delta(U_{\OO})$,
for then $\phi$ maps $U_{\OO}'$ into 
$U_{\OO}'\cap E^{-1}\delta^{\vee}\delta(U_{\OO})$.
In other words, if $\phi$ is zero in $\T_{\OO}^{\new}$,
then it is zero on $U_{\OO}''$.

Note first that $\phi^{\vee}$ maps $V_{\OO}$ into
$V_{\OO}^{\old}$. Indeed, for any $v\in V$, $v_n\in V^{\new}$
we have $\la v_n,\phi^\vee v\ra_V=\la\phi v_n,v\ra_V=0$, thus 
$\phi^\vee v\in V^{\old}$, so $\phi^{\vee}$ maps $V$ to $V^{\old}$.
Moreover $\phi\in H_{\OO}$ implies $\phi^{\vee}\in H_{\OO}$, thus 
$\phi^{\vee}$ maps $V_{\OO}$ to itself, so to $V_{\OO}^{\old}$.

Since $\phi^\vee$ also maps $V_{\OO}^{\new}$ to itself, it maps
$V_{\OO}^{\new}$ to zero. 

Now let $u =\delta^{\vee}(v) \in U_{\OO}$ for some 
$v \in V^{\old}$. Note that $\delta^{\vee}V\subset(\ker\delta)^{\bot}$,
thus $u\in U_{\OO}'$. We have 
$$
A_UC\la v,V_{\OO}^{\old} \ra_V \subset A_U\la
v,\delta(U_{\OO}) \ra_V \subset A_U\la u,U_{\OO}
\ra_U \subset \OO,
$$
from $CV_\OO^{\old}\subset\delta(U_\OO)$, $u =\delta^{\vee}(v)$, 
$A_U U_\OO \subset U_\OO^\vee$. Since $\phi^{\vee} V_\OO \subset V_\OO^{\old}$, 
we have $A_UC\la \phi v,V_{\OO} \ra_V \subset \OO$. By definition of $B_V$,
we deduce that $A_UB_VC(\phi v) \in V_{\OO}$, hence it is in $V_{\OO}^{\old}$,
as $v\in V^{\old}$ and thus $\phi v\in V^{\old}$. We conclude from the 
definition of $C$ that
$$
A_UB_VC^2(\phi v) \in \delta(U_{\OO}).
$$
We get the result by applying $\delta^{\vee}$ to this: 
$E(\phi u)\in \delta^{\vee}\delta(U_{\OO})$, as $u=\delta^\vee v$. Thus
$\phi$ takes $U'_\OO$ to $E^{-1}\delta^{\vee}\delta(U_{\OO})\cap U'_\OO$.
\end{proof}

As in [T2, p. 331], we have the following useful corollary:

\begin{cor}
Let $\OO=\OO_L$ be the ring of integers of a number field 
$L \subset \C$, or $\OO=R_{\lambda_0}$ and $L=L_{1\lambda_0}$. Suppose 
$u \in U_{\OO}-\{0\}$ is an eigenvector for $H_{\OO}$, with character 
$\eta:H_{\OO} \to \OO$. Define 
$\EE =\{x\in\OO; x(Lu \cap (U_{\OO}+\ker\delta))\subset\OO u\};$ 
it is an ideal in $\OO$. Suppose it is nonzero, and that 
$\delta^{\vee}\delta(u) \in m U_{\OO}$, for some nonzero $m \in \OO$.
Then $\eta$ induces a homomorphism $\T_{\OO}^{\new}
\to \OO/(\OO \cap mE^{-1}\EE^{-1})$.
\end{cor}

\begin{proof} 
Consider the action of $H_\OO$ on 
$u+Lu\cap U'_\OO\cap E^{-1}\delta^\vee\delta U_\OO$.
If $u_1^\bot\in U_\OO$ and $u_1^\bot\bot u$, and 
$\delta^\vee\delta u_1^\bot\in Lu$,
then $0=\la u_1^\bot,\delta^\vee\delta u_1^\bot\ra_V=
\la\delta u_1^\bot,\delta u_1^\bot\ra_V$.
Hence $\delta u_1^\bot=0$, thus $u_1^\bot\in\ker\delta$, and 
$Lu\cap\delta^\vee\delta U_\OO=Lu\cap\delta^\vee\delta (U_\OO\cap Lu)$. 
Also $U'_\OO=U_\OO\cap(\ker\delta)^\bot$, and 
$\im\delta^\vee\subset(\ker\delta)^\bot$.
Thus $Lu\cap U'_\OO\cap E^{-1}\delta^\vee\delta U_\OO
=\EE^{-1}u\cap E^{-1}\delta^\vee\delta(U_\OO\cap Lu)
=\EE^{-1}u\cap\EE^{-1}E^{-1}mU_\OO$. Now 
$\eta:H_\OO\to\OO$ is defined by $h\cdot u=\eta(h)u$, 
$h\in H_\OO$, so the action of $h\in H_\OO$ on a vector in 
$\EE^{-1}u\cap\EE^{-1}E^{-1}mU_\OO$ is by 
multiplication by $\eta(h)\in\OO\cap\EE^{-1}E^{-1}m$, 
thus $\eta$ induces a homomorphism as desired.
\end{proof} 

We remark that $m=0$ implies that  $\delta^\vee\delta u=0$, 
thus $\delta u=0$ since $\delta^\vee$ is injective on $V^{\old}$, 
so $u \in \ker\delta$. When $\OO=R_{\lambda_0}$ and $L=L_{1 \lambda_0}$
we get $\T_{\OO}^{\new} \to \OO/\lambda_0^n$ for
every positive 
$n\le n_0=v_{\lambda_0}(m)-v_{\lambda_0}(E)-v_{\lambda_0}(\EE)$.
If $\OO$ is the ring of integers in a number field $L$, and we 
factor the fractional ideal $\OO \cap mE^{-1}\EE^{-1}$
into prime powers and project further, we get the following. For
every (nonzero) prime ideal $\lambda \subset \OO$ there is
a homomorphism
$$
\T_{\OO}^{\new} \to \OO/\lambda^n
$$
induced by $\eta$, where $n$ is a nonnegative integer $\le n_0
= v_{\lambda}(m)-v_{\lambda}(E)-v_{\lambda}(\EE).$
Here we should think of $v_{\lambda}(m)$ as the main term, and the
other two as controllable error terms. In our applications 
$\OO=R_{\lambda_0}$ and $L=L_{1 \lambda_0}$. We want
to show that $n_0$ is positive.

\section{Compactness at Infinity}

Let $F$ be a totally real number field. Let $\infty$ be the
set of archimedean places. Denote the ring of ad\`eles by
$\A=\A_F=F_{\infty}\times \A^{\infty}$. Consider a connected
reductive $F$-group $G$. 
Each $F$-rational character $\chi\in X^*(G)_F$ gives a 
continuous homomorphism $G(\A)\to\R_+^\times$ by 
composing with the id\`ele norm. Define
$$G(\A)^1=\{g \in G(\A);|\chi(g)|=1, \forall \chi \in X^*(G)_F\}.$$
This group is known to be unimodular. By the product formula, 
$G(F)$ is a discrete subgroup of $G(\A)^1$. The quotient 
$G(F) \backslash G(\A)^1$ has finite volume. This quotient is 
compact if and only if $G^{\ad}$ is anisotropic. 
We shall naturally be led to studying groups for which
$G_{\infty}^1=G_{\infty}\cap G(\A)^1$ is compact.
Let $G^{\ab}$ denote the biggest quotient group of $G$ 
which is a torus, namely $G^{\ab}=G/G^{\der}$ ($G/G^{\der}$
is connected as it is the quotient of a connected group, $G$,
it is reductive since $G$ is, and it is abelian as it is the quotient
by $G^{\der}$, hence it is a torus).

\begin{prop}
The group $G_{\infty}^1$ is compact if and only if $G_{\infty}$ is compact, 
or$:$ $F=\Q$, $\rk_{\Q}G^{\ab}=\rk_{\R}G^{\ab}$, and $G_{\infty}^{\der}$
is compact.
\end{prop}

\begin{proof} 
Suppose first that $G_{\infty}^1$ is compact. Let $A$ denote the 
biggest quotient group of $G^{\ab}$ which is a split torus. Then
$X^\ast(G)_F = X^\ast(G^{\ab})_F = X^\ast(A)_F = X^\ast(A)$. 
Set $G'=\ker[G\to A]$. Since we have an exact sequence
$1\to G'_\infty \to G^1_\infty \to A^1_\infty \to 1$, we see that both  
$G'_\infty$ and $A^1_\infty$ are compact. We may assume that 
$A\neq 1$ (otherwise $G_\infty=G^1_\infty$, hence $G_\infty$ is compact).
Choosing a basis for $X^\ast(A)$, we see that (with $r=\dim A$)
$$
A_{\infty}^1 \simeq \{x \in F_{\infty}^\times; \prod_{v\in \infty}|x_v|_v=1\}^r.
$$
Therefore $\{x \in F_{\infty}^\times; \prod_{v|\infty} |x_v|_v=1\}$ is
compact. We conclude that $F$ has a unique infinite place. That is, 
$F=\Q$.  If $\rk_{\Q}G^{\ab} < \rk_{\R}G^{\ab}$, the $\Q$-anisotropic 
component $N(=(G^{\ab})'=\ker[G^{\ab}\to A])$ of $G^{\ab}$ is not 
$\R$-anisotropic, hence $N_\infty$ is not compact, so $G'_\infty$ 
is not compact and thus $G^1_\infty$ is not compact. The converse is clear.
\end{proof}

\section{Hecke Algebras}

From now on we fix a totally real number field $F$, and a
connected reductive $F$-group $G$, not a torus, such that
$G_{\infty}^1$ is compact. Consider the locally profinite group
of finite ad\`eles $G(\A^{\infty})$. Let $\Sigma$ be a finite set 
of finite places. Let $\A^\Sigma$ be the ring of ad\`eles without
component at $\infty$ and $\Sigma$. Consider the subgroup 
$G(\A^\Sigma)$ of $G(\A^\infty)$. Choose a Haar measure
$\mu=\otimes \mu_v$ on $G(\A^{\Sigma})$ once and for all. 
Consider the vector space of all locally constant compactly 
supported $\C$-valued functions on $G(\A^{\Sigma})$:
$$
\HH=\HH(G(\A^{\Sigma}))=C_c^{\infty}(G(\A^{\Sigma}),\C).
$$
This becomes an associative $\C$-algebra, without neutral element,
under $\mu$-convolution. There is a canonical involution 
(anti-automorphism) on $\HH$ defined by 
$\phi^{\vee}(x)=\phi(x^{-1})$. If $K \subset G(\A^{\Sigma})$ is a 
compact open subgroup, $e_K =\mu(K)^{-1}\chi_K \in \HH$
is an idempotent. This is the neutral element in the subalgebra of
$K$-biinvariant compactly supported functions:
$$
\HH_K=\HH(G(\A^{\Sigma}),K)=C_c(G(\A^{\Sigma})//K,\C)
=e_K\ast \HH \ast e_K.
$$
Clearly $\vee$ preserves $\HH_K$. In addition, there is a
canonical $\Z$-order $\HH_{K,\Z} \subset \HH_K$
consisting of all $\mu(K)^{-1}\Z$-valued functions. As a ring,
$\HH_{K,\Z}$ is isomorphic to $C_c(G(\A^{\Sigma})//K,\Z)$
endowed with the $K$-normalized convolution. If $R$ is a
commutative ring, with neutral element, we define
$$
\HH_{K,R}=\HH_{K,\Z}\otimes_{\Z}R,\quad 
\text{thus}\quad\HH_{K}=\HH_{K,\C}.
$$
The algebras $\HH_K$ are not commutative when $K$ is not 
maximal. However, by a result of Bernstein [B], $\HH_K$ is a 
finite module over its center $Z(\HH_K)$. Now, suppose 
$J \subset K$ is a (proper) compact open subgroup. Then obviously 
$\HH_K\subset \HH_J$. However, $\HH_K$ 
is not a subring since $e_K \neq e_J$. There is a natural retraction 
$\HH_J\onto \HH_K$ defined by 
$\phi \mapsto e_K\ast \phi \ast e_K$. It does map $e_J \mapsto e_K$, 
but does not preserve $\ast$ unless we restrict it to the centralizer
$Z_{\HH_J}(e_K)$. Clearly,
$Z_{\HH_J}(\HH_K)$ maps to the center
$Z(\HH_K)$. In particular,
$$
\text{$Z(\HH_J) \to Z(\HH_K)$, $\y$ 
$\phi \mapsto \phi \ast e_K = e_K \ast \phi$,}
$$
gives a canonical homomorphism of algebras. It maps
$Z(\HH_{J,\Z})$ into $Z(\HH_{K,\Z})$.

\section{Algebraic Modular Forms}

In this section we define algebraic modular forms with weight and type,
using the exposition of Bellaiche and Graftieaux [BG].
For each finite place $v$ of $F$, let $K_v$ be a compact open subgroup
of $G_v=G(F_v)$, which is a maximal compact hyperspecial subgroup
for almost all $v$. Let $\Sigma$ be a finite set of finite places of $F$.
Write $K_\Sigma$ for $\prod_{v\in\Sigma}K_v$ and 
$K=\prod_{v\notin\Sigma}K_v$. Put $K'{}'=K\times K_\Sigma$. Then
$K'{}'=\prod_vK_v$ ($v<\infty$) is a compact open subgroup of $G(\A^\infty)$.
It is called the {\it level}.  Let $\rho_\Sigma:K_\Sigma\to\GL(\Vs)$ be a smooth 
complex irreducible representation, named the {\it type}. It can be viewed 
as a representation of $K'{}'$ trivial on $K$. 

Put  $G_\infty=G(F_\infty)=\prod_vG(F_v)$, $F_\infty=\prod_vF_v$ 
($v$ archimedean).  Let $\rhoi:G_\infty\to\GL(\Vi)$ be an irreducible, 
complex, continuous unitary representation, named the {\it weight}.
Denote by $Z_\infty$ the center of $G_\infty$, and by $\omega_\infty$
the central character of $\rho_\infty$.

Denote by $\rhoi^\ast$ and $\rhos^\ast$ the contragredient representations.

Note that $G(F) \subset G(\A)/Z_\infty$ is a discrete cocompact subgroup. 
Consider  the Hilbert space $L^2(G(F) \backslash G(\A),\omega_\infty)$ of 
$L^2$-functions on the quotient $G(F) \backslash G(\A)$ which transform 
under $Z_\infty$ via the unitary character $\omega_\infty$. There is a unitary 
representation $r$ of $G(\A)$ on this space given by right translations. This 
space is a direct sum, with finite multiplicities $m(\pi)$, of irreducible 
$G(\A)$-submodules $\pi$, called {\it automorphic representations}. An 
admissible irreducible representation $\pi$ decomposes as a product 
$\otimes_v\pi_v$ over all places $v$ of $F$. Put 
$\pi_\Sigma=\otimes_{v\in\Sigma}\pi_v$.  We shall
be interested only in the part which contains the representation
$\rho=\rho_\Sigma\otimes\rho_\infty$ of $K_\Sigma\times G_\infty$.
It is $\aA(K_\Sigma,\rho,\C)=\cup_{K}\aA(K,\rho,\C)$
$$=\Hom_{K_\Sigma\times G_\infty}(\rho_\Sigma\otimes\rho_\infty,
C^\infty(G(F)\bs G(\A),\omega_\infty;\C)).$$
Here $K$ runs through
all compact open subgroups of $G(\A^\Sigma)$, and
$$\aA(K,\rho,\C)=\Hom_{K_\Sigma\times G_\infty}
(\rho_\Sigma\otimes\rho_\infty,C^\infty(G(F)\bs G(\A)/K,\omega_\infty;\C)).$$
The Hecke algebra $\HH=\HH(G(\A^\Sigma))$ acts on
$\aA(K_\Sigma,\rho,\C)$ by convolution, and $\aA(K,\rho,\C)$
is the space $r(e_{K})\aA(K_\Sigma,\rho,\C)$ of
$K$-invariants. It is a finite dimensional space, as the 
double coset space $X_K=G(F)\bs G(\A^{\infty})/K'{}'$ is finite.
Recall that $K'{}'=K\times K_\Sigma$.
Thus $\aA(K_\Sigma,\rho,\C)=\oplus m(\pi)\pi$, sum over
the irreducible $\pi$ with $\pi_\infty=\rho_\infty$, $\pi_\Sigma\supset\rho_\Sigma$.
Also $\aA(K,\rho,\C)=\oplus m(\pi)\pi^K$, sum over the same $\pi$,
but for which the space $\pi^K$ of $K$-invariants in $\pi$ is nonzero.

We have the following compatibility between this
action and the inner product:
$$
(r(\bar{\phi})f,g)=(f,r(\phi^{\vee})g);\qquad\phi\in\HH; 
\,\,\,f,\,g\in \aA(K_\Sigma,\rho,\C).
$$

\begin{dfn}
A complex valued automorphic form of level $K'{}'$, type $\rhos$, and weight
$\rhoi$, is a function $f''$ in $\aA(K,\rho,\C)$ where $\rho=(\rhos,\rhoi).$

Using the relation $[f'(g)](v,w)=[f''(v,w)](g)$, the space 
$\aA(K,\rho,\C)$ is isomorphic to the space of functions 
$f':G(F)\bs G(\A)/K\to\Vs^\ast\otimes_\C\Vi^\ast$ with
$$f'(gku_\infty)=[\rhos^\ast(k)^{-1}\otimes\rhoi^\ast(u_\infty)^{-1}]f'(g)\quad 
(g\in G(\A),\,k\in K'{}',\, u_\infty\in G_\infty).$$

The restriction $f$ of such $f'$ to $G(\A^\infty)$ satisfies, where $u\in G(F)$, 
$u_\infty$ is the image of $u$ in $G_\infty$, so that 
$uu_\infty^{-1}\in G(F)\cap G(\A^\infty)$, the relation
\begin{equation}
\label{eq1} f(ugk)\,\, [=f'(ugku_\infty^{-1})=f'(gku_\infty^{-1})]= 
(\rhos^\ast(k)^{-1}\otimes\rhoi^\ast(u_\infty))f(g),
\end{equation}
where $g\in G(\A^\infty)$, $k\in K'{}';$ it determines $f'$.
\end{dfn}

Since $\rhos$ factorizes through a finite quotient of $K'{}'$, there exists a
{\it number field} $L$ (in $\C$) and an $L$-model $(\rhosL,\VsL)$ of $(\rhos,\Vs)$.
Increasing $L$ we may assume it is stable under complex conjugation. We can 
then talk about Hermitian products on $\VsL$. Choose such an Hermitian
product on $\VsL$ which is stable under $\rhosL$ (using the finiteness of
$\rhosL(K)$).

Extend the representation $\rhoi$ of the algebraic group 
$G_\infty=G(\R)^{\Sigma_\infty}$
on $\Vi$ to a representation $\rho_{\infty,\C}$ of 
$G(\C)^{\Sigma_\infty}$ on the
same complex space $\Vs$. Since $\rho_{\infty,\C}$ is algebraic, it has a model
over a number field, which can be assumed to be $L$ (on increasing $L$, so
that in particular it contains $F$). Thus $\rhoiL$ is an algebraic 
representation of $G(L)^{\Sigma_\infty}$ on $\ViL$, defined by a morphism 
$G^{\Sigma_\infty}\otimes_FL\to\GL(\ViL)$ of group schemes over $L$, 
which specializes to $\rho_{\infty,\C}$ over $\C$. As usual, $\otimes_FL$
abbreviates $\times_{\Spec F}\Spec L$.

Increasing $L$ we may assume it is Galois over $\Q$. Then $\sigma(F)\subset L$
for each $\sigma$ in $\Sigma_\infty$. Embed $G(F)$ in $G(L)^{\Sigma_\infty}$ by
$r:u\mapsto (\sigma(u);u\in\Sigma_\infty)$. Then by definition, for every 
$u\in G(F)$ we have
\begin{equation}
\label{eq2}   \rhoiL(r(u))\otimes 1=\rho_\infty(u_\infty)\in\GL(\Vi),
\end{equation}
the tensor product is $\ViL\otimes_L\C=\Vi$.

\begin{dfn}
For any commutative unitary $L$-algebra $A$, put 
$V_{\Sigma,A}=\VsL\otimes_LA$, $V_{\infty,A}=\ViL\otimes_LA$. 
Let $\rho_{\Sigma,A}$, $\rho_{\infty,A}$ be the corresponding 
representations. Let $\aA(K,\rho,A)$ be the $A$-module of smooth functions
\begin{equation}
\label{eq3}     f:G(\A^\infty)\to V_{\Sigma,A}^\ast\otimes_A V_{\infty,A}^\ast
\end{equation}
satisfying for $g\in G(\A^\infty)$, $u\in G(F)$, $k\in K'{}'$,
$$f(ugk)=[\rho_{\Sigma,A}^\ast(k)^{-1}\otimes\rho_{\infty,A}^\ast(r(u))]f(g).$$
\end{dfn}

When $A=\C$, (5.3) coincides with (5.1) in view of (5.2).

Since $G_\infty$ is compact, $\rhoi$ preserves an Hermitian form on $\Vi$.
The restriction of $\rhoiL$ to $G(L_0)^{\Sigma_0}$ does too. 
Here $L_0=L\cap\R$, and we assume $L\not=L_0$, thus $[L:L_0]=2$. 
Choose such an Hermitian form on $\ViL$. Given $u\in G(F)$, for any 
embedding $\sigma$ of $F$ in $\C$, we have $\sigma(u)\in G(L_0)$ 
since $F$ is totally real. Hence $\rhoiL(r(u))$ is a unitary element 
in $\GL(\ViL)$. The space $G(F)\bs G(\A^\infty)$ is compact, thus 
there is a unique right invariant measure on it which assigns it 
volume one. We obtain an Hermitian form on $\aA(K,\rho,L)$ on integrating 
over $G(F)\bs G(\A^\infty)$ that on $\VsL^\ast\otimes_L\ViL^\ast$.

\begin{dfn}
Let $\lambda$ be a finite place of $L$. Let 
${\lambda_0}=\lambda\cap L_0$ be the place 
of $L_0=L\cap\R$ dividing $\lambda$. Let 
$(L_0)_{{\lambda_0}}$ be the completion 
of $L_0$ at ${\lambda_0}$. Put 
$L_{\lambda_0}=L\otimes_{L_0}(L_0)_{{\lambda_0}}$. Let $R_{\lambda_0}$ be
the ring of integers of $L_{\lambda_0}$. If ${\lambda_0}$ splits in $L$ into 
$\lambda\cdot\ov\lambda$,
then $R_{\lambda_0}=R_\lambda\times R_{\ov\lambda}$. If ${\lambda_0}$ 
stays prime, $R_{\lambda_0}$ is $R_\lambda$
$([L_{\lambda_0}:L_{0,{\lambda_0}}]=2$, then $L_{\lambda_0}
=L_\lambda$ and $R_{\lambda_0}=R_\lambda)$. Then $R_{\lambda_0}$
has an involution extending that on the ring of integers $R_L$ of $L$.

Let $\F_\lambda=R_\lambda/\lambda$ be the residue field of 
$R_\lambda$ $(\F_{\ov\lambda}$ of $R_{\ov\lambda})$.
Put $\F_{\lambda_0}=R_{\lambda_0}/{\lambda_0}$ $($it is 
$\F_\lambda\times\F_{\ov\lambda}$ if ${\lambda_0}$ splits, or 
$\F_\lambda$ if ${\lambda_0}$ stays prime$)$. Note that 
$\F_{\ov\lambda}\simeq\F_\lambda$.
 
Write $\Sigma_{\lambda_0}$ for the set of places of $F$ of the same residual
characteristic as that of ${\lambda_0}$. 
\end{dfn}

For $\sigma\in\Sigma_\infty$, 
$\sigma:F\into L_0$, we have $\sigma^{-1}({\lambda_0})\in\Sigma_{\lambda_0}$ 
(as ${\lambda_0}\subset R_0\subset L_0$). The embedding $\sigma$ extends by 
continuity to an embedding $\sigma:F_{\sigma^{-1}({\lambda_0})}
\into L_{0\lambda_0}\into L_{\lambda_0}$.
Then we have $r_\sigma:G(F_{\sigma^{-1}({\lambda_0})})\into G(L_{\lambda_0})$ 
for each $\sigma\in\Sigma_\infty$, and the product 
$r_{L_{\lambda_0}}:G(F_{\Sigma_{\lambda_0}})
\into G(L_{\lambda_0})^{\Sigma_\infty}$, 
$$(x_{\sigma^{-1}({\lambda_0})};\sigma\in\Sigma_\infty)
\mapsto(r_\sigma(x_{\sigma^{-1}({\lambda_0})});\sigma\in\Sigma_\infty).$$

Let $p_{\Sigma_{\lambda_0}}:G(\A^\infty)\to G(F_{\Sigma_{\lambda_0}})$ 
be the natural projection. Put 
$$\wt{\rho}_{\infty,L_{\lambda_0}}
=\rho_{\infty,L_{\lambda_0}}\circ r_{L_{\lambda_0}}\circ 
p_{\Sigma_{\lambda_0}}: G(\A^\infty)\to\GL(V_{\infty,L_{\lambda_0}}).$$
It has, for $u\in G(F)$, that 
$\wt{\rho}_{\infty,L_{\lambda_0}}(u)\in\GL(V_{\infty,L_{\lambda_0}})$ 
is in fact in $\GL(V_{\infty,L})$, equal to
$$\wt{\rho}_{\infty,L_{\lambda_0}}(u)=\rhoiL(r(u))=\rhoi(u_\infty)\qquad 
(u\in G(F)).$$

\begin{dfn}
Fix any $R_L$-lattices $V_{\Sigma,R_L}$ in $\VsL$ and $V_{\infty,R_L}$ 
in $V_{\infty,L}$. Put $V_{\Sigma,R_{\lambda_0}}=
V_{\Sigma,R_L}\otimes_{R_L}R_{\lambda_0}$, 
$V_{\infty,R_{\lambda_0}}=V_{\infty,R_L}\otimes_{R_L}R_{\lambda_0}$.
Let $\aA(K,\rho,R_{\lambda_0})$ be the sub-$R_{\lambda_0}$-module 
of $\aA(K,\rho,L_{\lambda_0})$ consisting of the functions 
$$f:G(\A^\infty)\to V_{\Sigma,L_{\lambda_0}}^\ast\otimes_{L_{\lambda_0}}
V_{\infty,L_{\lambda_0}}^\ast$$ with 
$$\wt{f}(g)=\wt{\rho}_{\infty,L_{\lambda_0}}^\ast(g)^{-1}f(g)
\in V_{\Sigma,R_{\lambda_0}}^\ast\otimes_{R_{\lambda_0}}
V_{\infty,R_{\lambda_0}}^\ast$$ 
for all $g$ in $G(\A^\infty)$. Put $K^{\lambda_0}=\prod_vK_v$ 
$(v\notin\Sigma\cup\Sigma_{\lambda_0}\cup\infty)$. 
\end{dfn}

Note that $\wt{\rho}_{\infty,L_{\lambda_0}}$ factorizes through
$G(F_{\Sigma_{\lambda_0}})$, thus $\aA(K,\rho,R_{\lambda_0})$ 
can be viewed as a set of
$K^{\lambda_0}$-invariants of a space of functions on $G(\A^\infty)$
on which $G(\A^{\infty})$ acts by right translation.
We aim to show, in Lemma 5.6, that for almost all places $\lambda$, we have
$\aA(K,\rho,R_{\lambda_0})\otimes_{R_{\lambda_0}}L_{\lambda_0}
=\aA(K,\rho,L_{\lambda_0})$.

The morphism $\rhoiL:G^{\Sigma_\infty}\otimes_FL\to\GL(\ViL)$ extends -- 
since $G$ and $\GL(\ViL)$
are schemes of finite type -- to a morphism
$$\GG^{\Sigma_\infty}\otimes_{R_F}R_L[1/N]\to\GL(V_{\infty,R_L[1/N]})$$
over the open subset $\Spec R_L[1/N]$ of $\Spec R_L$, where $N$ is a positive
integer and $\GG$ is a smooth affine group scheme of finite type over $R_F$
with generic fiber $G$. For each ${\lambda_0}$ not dividing $N$ there is then a model
$$\rho_{\infty,{\lambda_0}}:\GG^{\Sigma_\infty}\otimes_{R_F}
R_{\lambda_0}\to\GL(V_{\infty,R_{\lambda_0}})$$
over $R_{\lambda_0}$ of $\rho_{\infty,L_{\lambda_0}}$. The lattice 
$V_{\infty,R_{\lambda_0}}$ of $V_{\infty,L_{\lambda_0}}$
is stable under the restriction of $\rho_{\infty,L_{\lambda_0}}$ to the subgroup 
$\GG^{\Sigma_\infty}(R_{\lambda_0})$ of $\GG^{\Sigma_\infty}(L_{\lambda_0})$.
Increasing $N$ we may also assume that $\GG(R_v)=K_v$ for each
$F$-place $v$ prime to $N$. Then for each ${\lambda_0}$ prime to $N$ the 
morphism $r_{L_{\lambda_0}}\circ p_{\Sigma_{\lambda_0}}$ maps $K'{}'$ to 
$\GG^{\Sigma_\infty}(R_{\lambda_0})$. Hence the restriction of
$\wt{\rho}_{\infty,L_{\lambda_0}}$ to $K'{}'$ leaves stable the lattice 
$V_{\infty,R_{\lambda_0}}$ of $V_{\infty,L_{\lambda_0}}$. 
For each $R_{\lambda_0}$-algebra $A$ we can then define the 
representation $\wt{\rho}_{\infty,A}$ of the group $K'{}'$ on 
$V_{\infty,A}=V_{\infty,R_{\lambda_0}}\otimes_{R_{\lambda_0}}A$.
In particular one has a representation $\wt{\rho}_{\infty,\F_{\lambda_0}}$ 
of $K'{}'$ on $V_{\infty,\F_{\lambda_0}}$.

\begin{lem} Increasing $N$ if necessary, for each place 
${\lambda_0}$ of $L_0$ prime to $N$, the restriction of 
$\wt{\rho}_{\infty,\F_{\lambda_0}}$ to any subgroup $H$ of $K'{}'$ 
whose image under reduction to
$\prod_{v\in\Sigma_{\lambda_0}}\GG(\F_v)$ contains
$\prod_{v\in\Sigma_{\lambda_0}}\GG^{\der}(\F_v)$ is
absolutely irreducible.
\end{lem}

\begin{proof}
The embedding $r:G(F)\to G^{\Sigma_\infty}(L)$ is the
realization for $\Q$-points of a morphism 
$\text{R}_{F/\Q}G\to\text{R}_{L/\Q}G^{\Sigma_\infty}$
of algebraic $\Q$-groups. There exists a model of this morphism 
over $\Spec\Z[1/N]$ for a suitable $N$. For a prime $\ell$ prime
to $N$ one has the morphism
$$r_{\Z_\ell}:\prod_{v|\ell}\GG(R_v)\to\prod_{{\lambda_0} |\ell}
\GG^{\Sigma_\infty}(R_{\lambda_0})$$
($v$ are $F$-places, ${\lambda_0}$ are $L_0$-places). There are analogous
morphisms $r_{\Q_\ell}$ with $R_v$ replaced by $F_v$ and $R_{\lambda_0}$ by
$L_{\lambda_0}$, and $r_{\F_\ell}$ with $\F_v$ and 
$\F_{\lambda_0}$ (for $R_v$, $R_{\lambda_0}$).
Note that $\{v;v|\ell\}$ is $\Sigma_{\lambda_1}$ 
for any $\lambda_1$ dividing $\ell$.
The morphism $r_{L_{\lambda_1}}$ is then $r_{\Q_\ell}\circ\pr$, $\pr$ being
the projection of $\prod_{{\lambda_0} |\ell}\GG^{\Sigma_\infty}(L_{\lambda_0})$
to its factor $\GG^{\Sigma_\infty}(L_{\lambda_1})$.

Since $G(F)$ is Zariski dense in $G(\C)^{\Sigma_\infty}$, the morphism
$$\text{R}_{F/\Q}G\to \text{R}_{L/\Q}G^{\Sigma_\infty}\quad
\stackrel{\text{R}_{L/\Q}\rhoiL}{\longrightarrow}\quad\text{R}_{L/\Q}\GL(\ViL)$$
is absolutely irreducible. The same holds with $G$ replaced by $G^{\der}$,
since $G=G^{\der}\cdot Z$, $Z$ being the center of $G$. The same holds
for almost all $\ell$ and ${\lambda_0}$ dividing $\ell$, for the morphism
$${\psi_{\lambda_0}:\prod_{v\in\Sigma_{\lambda_0}}\GG(\F_v)
\stackrel{r_{\F_\ell}}{\longrightarrow}
{\prod_{{\lambda_0}|\ell}\GG^{\Sigma_\infty}(\F_{\lambda_0})}
\stackrel{\pr}{\longrightarrow}
\GG^{\Sigma_\infty}(\F_{\lambda_1})\,\,\,
\stackrel{\rho_{\infty,\F_{\lambda_1}}}{\longrightarrow}
\quad\GL(V_{\infty,\F_{\lambda_1}}),}$$
and also with $\GG^{\der}$ replacing $\GG$.

The lemma follows from the commutativity of the lower triangle in the following
diagram, where the square is clearly commutative: 

$$\begin{matrix}
\prod_{v\in\Sigma_{\lambda_1}}K_v   &
\stackrel{\wt{\rho}_{\infty,R_{\lambda_1}}} {\longrightarrow} &
 \GL(V_{\infty,R_{\lambda_1}})\\
\downarrow{\text{reduction}}
&\searrow~~\wt{\rho}_{\infty,\F_{\lambda_1}}
&\downarrow{\text{reduction}}\\
\prod_{v\in\Sigma_{\lambda_1}}\GG(\F_v)&\stackrel{\psi_{\lambda_1}}
{\longrightarrow}&\GL(V_{\infty,\F_{\lambda_1}});
\end{matrix}$$
the upper triangle is commutative by the definition of 
$\wt{\rho}_{\infty,\F_{\lambda_1}}$.
\end{proof}

Since $\rho_\Sigma(K'{}')$ is finite, increasing $N$ we may assume
that for ${\lambda_0}$ prime to $N$ the lattice $V_{\Sigma,R_{\lambda_0}}$ is
stable under $\rho_\Sigma(K)$, and the representation $\rho_\Sigma$
of $K'{}'$ on $V_{\Sigma,R_{\lambda_0}}/{\lambda_0} V_{\Sigma,R_{\lambda_0}}
=V_{\Sigma,\F_{\lambda_0}}$ is absolutely irreducible.

As $G(F)\bs G(\A^\infty)$ is compact, there are $x_1,\dots,x_h$ in
$G(\A^\infty)$ with $G(\A^\infty)$ $=\coprod_{1\le i\le h}G(F)x_iK'{}'$.
The group $\Delta_i=G(F)\cap x_iK'{}'x_i^{-1}$ is finite ($1\le i\le h$),
$G(F)\bs G(\A^\infty)=\coprod_{1\le i\le h}\Delta_i\bs x_iK'{}'$.
The map
$$f\mapsto(f(x_1),\dots,f(x_h)),\quad \aA(K,\rho,A)\to\oplus_{1\le 
i\le h}(V_{\Sigma,A}^\ast\otimes_A V_{\infty,A}^\ast)^{\Delta_i},$$
where $\alpha=x_ikx_i^{-1}\in\Delta_i$ acts by 
$\rho_{\Sigma,A}^\ast(k)^{-1}\otimes_A \rho_{\infty,A}^\ast(r(\alpha))^{-1}$,
is an isomorphism. The Hermitian product on $\aA(K,\rho,A)$ 
is the sum over $i$ ($1\le i\le h$) of $|\Delta_i|^{-1}$ times that on the 
spaces $(V_{\Sigma,L}^\ast\otimes_LV_{\infty,L}^\ast)^{\Delta_i}$. Thus
$$\mu(K'{}')^{-1}(f,g)=\sum_{1\le i\le h}(f(x_i),g(x_i))|\Delta_i|^{-1}.$$
So $\aA(K,\rho,A')=\aA(K,\rho,A)\otimes_AA'$ for any 
$A\to A'$. As $Z(\HH_{K})$-modules,
$\aA(K,\rho,L)\otimes_L\C$ $=\aA(K,\rho,\C)$. 
The Hermitian product on $\aA(K,\rho,L)$ is nondegenerate: the adjoint
of the action of an element of $Z(\HH_{K})$ on $\aA(K,\rho,L)$ is
still the action of an element of $Z(\HH_{K})$. This last algebra is 
commutative, thus the elements of $Z(\HH_{K})$ act as normal 
($DD^\ast=D^\ast D$) operators, and $\aA(K,\rho,L)$ is a semisimple 
$Z(\HH_{K})\otimes L$-module. The sub-$R_L$-module 
$(V_{\Sigma,R_L}^\ast\otimes_{R_L}V_{\infty,R_L}^\ast)^{\Delta_i}$
of the $L$-vector space 
$(V_{\Sigma,L}^\ast\otimes_LV_{\infty,L}^\ast)^{\Delta_i}$ is a lattice.
 
Increasing $N$ we may assume that each $|\Delta_i|$ divides $N$, and for any
${\lambda_0}$ prime to $N$, the restriction of the Hermitian product of 
$V_{\Sigma,L_{\lambda_0}}^\ast\otimes_{L_{\lambda_0}}
V_{\infty,L_{\lambda_0}}^\ast$ to $(V_{\Sigma,L_{\lambda_0}}^\ast
\otimes_{L_{\lambda_0}}V_{\infty,L_{\lambda_0}}^\ast)^{\Delta_i}$
is nondegenerate and $R_{\lambda_0}$-valued. Note: $N$ depends only on
$(K'{}',\rho)$. Replacing $K'{}'$ by a subgroup we need not change $N$.

\begin{lem} For any place ${\lambda_0}$ of $L_0$ prime to $N$ we have
$$\aA(K,\rho,R_{\lambda_0})\otimes_{R_{\lambda_0}}L_{\lambda_0}
=\aA(K,\rho,L_{\lambda_0})$$ as a
$Z(\HH_{K^{\lambda_0}})\otimes_\Z L_{\lambda_0}$-module. The
Hermitian product on $\aA(K,\rho,R_{\lambda_0})$ is nondegenerate.
\end{lem}

\begin{proof} 
To show that $\aA(K,\rho,R_{\lambda_0})$ is a lattice 
in $\aA(K,\rho,L_{\lambda_0})$,
note that for $f\in \aA(K,\rho,R_{\lambda_0})$, $g\in G(\A^\infty)$, $u\in G(F)$, 
$k\in K'{}'$, we have
$$\wt{f}(ugk)=\wt{\rho}_{\infty,L_{\lambda_0}}^\ast(ugk)^{-1}f(ugk)$$
by definition of $\wt{f}$,
$$=(\rho_\Sigma^\ast(k)^{-1}\otimes_{L_{\lambda_0}}
[\wt{\rho}_{\infty,L_{\lambda_0}}^\ast(k)^{-1}
\wt{\rho}_{\infty,L_{\lambda_0}}^\ast(g)^{-1}\wt{\rho}_{\infty,
L_{\lambda_0}}^\ast(u)^{-1}]\rho_\infty^\ast(u_\infty))f(g)$$
by definition of $f$,
$$=(\rho_\Sigma^\ast(k)^{-1}\otimes_{L_{\lambda_0}}
\wt{\rho}_{\infty,L_{\lambda_0}}^\ast(k)^{-1})
(\wt{\rho}_{\infty,L_{\lambda_0}}^\ast(g)^{-1}f(g))$$
since $\rho_\infty(u_\infty)=\wt{\rho}_{\infty,L_{\lambda_0}}(u)$,
$$=(\rho_\Sigma^\ast(k)^{-1}\otimes_{L_{\lambda_0}}
\wt{\rho}_{\infty,L_{\lambda_0}}^\ast(k)^{-1})\wt{f}(g)$$
by definition of $\wt{f}.$

By definition of $\aA(K,\rho,R_{\lambda_0})$, $f$ lies in this space precisely when
$\wt{f}(x_i)=\wt{\rho}_{\infty,L_{\lambda_0}}^\ast(x_i)^{-1}f(x_i)$ lies in 
$V_{\Sigma,R_{\lambda_0}}^\ast\otimes_{R_{\lambda_0}}
V_{\infty,R_{\lambda_0}}^\ast$ for all $i$ ($1\le i\le h$). 
Thus we have an isomorphism
$$\aA(K,\rho,R_{\lambda_0})=\oplus_{1\le i\le h}(\wt{\rho}_{\infty,R_{\lambda_0}}^\ast(x_i)
[V_{\Sigma,R_{\lambda_0}}^\ast\otimes_{R_{\lambda_0}}
V_{\infty,R_{\lambda_0}}^\ast])^{\Delta_i}.$$

Each of the $R_{\lambda_0}$-modules on the right is a 
lattice in the $L_{\lambda_0}$-vector
space $(V_{\Sigma,L_{\lambda_0}}^\ast\otimes_{L_{\lambda_0}}
V_{\infty,L_{\lambda_0}}^\ast)^{\Delta_i}$.
It remains to show that the restriction of the Hermitian product from
$\aA(K,\rho,L_{\lambda_0})$ to $\aA(K,\rho,R_{\lambda_0})$ 
is $R_{\lambda_0}$-valued and 
nondegenerate. But this is explained in the paragraph before the lemma.
Indeed, this Hermitian product is a direct sum, weighted by invertible
elements of $R_{\lambda_0}$, of Hermitian products which are nondegenerate
and $R_{\lambda_0}$-valued.
\end{proof}

For a place ${\lambda_0}$ of $L_0$ prime to $N$, and commutative  
$R_{\lambda_0}$-algebra $A$ with a unit, put 
$V_{\Sigma,A}=V_{\Sigma,R_{\lambda_0}}\otimes_{R_{\lambda_0}}A$,
$\rho_{\Sigma,A}=\rho_{\Sigma,R_{\lambda_0}}\otimes_{R_{\lambda_0}}A$, and
similarly $V_{\infty,A}$, $\rho_{\infty,A}$. Define $\aA(K,\rho,A)$ 
to be the $A$-module of smooth functions $\wt{f}:G(\A^\infty)\to
V_{\Sigma,A}^\ast\otimes_AV_{\infty,A}^\ast$ satisfying, for 
$g\in G(\A^\infty)$, $u\in G(F)$, $k\in K'{}'$,
$$\wt{f}(ugk)=[\rho_{\Sigma,A}^\ast(k)^{-1}
\otimes_A\wt{\rho}_{\infty,A}^\ast(k)^{-1}]\wt{f}(g).$$
The $Z(\HH_{K^{\lambda_0}})$-module $\aA(K,\rho,A)$ commutes
with base change $A\to A'$. For $A=R_{\lambda_0}$, $\aA(K,\rho,A)$ is 
isomorphic to $\aA(K,\rho,R_{\lambda_0})$ previously defined, by 
$\wt{f}\mapsto f$, $\wt{f}(g)
=\wt{\rho}_{\infty,L_{\lambda_0}}^\ast(g)^{-1}f(g)$.
Lemma 5.6 implies that any character of $Z(\HH_{K^{\lambda_0}})$ on 
$\aA(K,\rho,R_{\lambda_0})$ is $R_{\lambda_0}$-valued for almost all 
places $\lambda$ of $L$.

Let $\T_{K^{\lambda_0},A}$
denote the image of the center $Z(\HH_{K^{\lambda_0},A})$ in $\End_A
\aA(K,\rho,A)$. Hence $\T_{K^{\lambda_0},A}$ is a commutative $A$-algebra.
Now, suppose $J \subset K$ is a (proper) compact open subgroup.
Then $\aA(K,\rho,A) \subset \aA(J,\rho,A)$, and the canonical
homomorphism $Z(\HH_{J^{\lambda_0},A}) \to
Z(\HH_{K^{\lambda_0},A})$ descends to the restriction map
$\T_{J^{\lambda_0},A}\to \T_{K^{\lambda_0},A}$.

\section{Pairings}

We review now the pairing on $\aA(K,\rho,A)$. Here $(-,-)$
denotes the inner product on $V_{\Sigma,R_{\lambda_0}}^\ast
\otimes V_{\infty,R_{\lambda_0}}^\ast$.

\begin{dfn}
For $f, g \in \aA(K,\rho,A)$, define a symmetric bilinear form by
$$
\la f,g \ra_K=\mu(K'{}')^{-1}(f,\ov{g})=\sum_{x \in X_K}
(f(x),\ov g(x))|G(F)\cap {^xK'{}'}|^{-1}.
$$
Here ${^xK'{}'}=xK'{}'x^{-1}$ and $X_K=G(F)\bs G(\A^\infty)/K'{}'$.
\end{dfn}

The factors $|G(F)\cap {^xK'{}'}|^{-1}$ are missing in [T1] and
[T2]. If $K$ is sufficiently small, for example if $K=\prod_{v \notin
\Sigma\cup\infty}K_v$ and $K_v$ is torsion-free for some $v \notin
\Sigma\cup\infty$ (this is the case if
$K_v$ is a sufficiently deep principal congruence subgroup), then
indeed $G(F)\cap {^xK'{}'}=1$. For $\phi \in \HH_K$ and $f, g
\in \aA(K,\rho,A)$ we have the compatibility relation
$$
\la r(\phi)f,g \ra_K=\la f, r(\phi^{\vee})g \ra_K.
$$

Next we have to show that the quotient
$\aA(K,\rho,R_{\lambda_0})/\aA(K,\rho,R_{\lambda_0})^{\vee}$ 
is torsion and find a good annihilator $A_K$. The fact that it is torsion is
immediate: it is killed by the positive integer
$$
\prod_{x \in X_K}|G(F)\cap {^xK'{}'}|.
$$
This is $1$ if $K$ is sufficiently small in the sense above.

\begin{lem}
Let $K=\prod_{v\notin\Sigma}K_v \subset G(\A^{\Sigma})$ be a
decomposable compact open subgroup. Let $\ell$ be the residual
characteristic of $\lambda$. Suppose $\ell \nmid |K_v|$ for some 
$v \notin\Sigma\cup \infty$. Then there exists a positive integer 
$A_K$, not divisible by $\ell$, such that
$$
A_K \la \aA(K,\rho,R_{\lambda_0}),\aA(K,\rho,R_{\lambda_0}) 
\ra_K \subset R_{\lambda_0}.
$$
\end{lem}

\begin{proof}
Choose some torsion-free subgroup $\wt{K}_v \subset
K_v$ and let $\wt{K}=\wt{K}_vK^v$. Then
$$
\la \aA(\wt{K},\rho,R_{\lambda_0}),\aA(\wt{K},\rho,R_{\lambda_0})
\ra_{\wt{K}} \subset R_{\lambda_0}
$$
as we have observed above. Thus, for $f, g \in
\aA(K,\rho,R_{\lambda_0}) \subset \aA(\wt{K},\rho,R_{\lambda_0})$, we have
$$
[K_v:\wt{K}_v]\la f,g \ra_K=\la f,g \ra_{\wt{K}} \in R_{\lambda_0}.
$$
We then take $A_K=[K_v:\wt{K}_v]$. This is not divisible by $\ell$. 
\end{proof}

Note that $\ell \nmid |K_v|$ if $K_v$ is torsion-free and $v \nmid
\ell$. For large $\ell$ this is automatic:

\begin{lem}
Given an $F$-embedding $G \hookrightarrow \GL(n)$, a compact open
subgroup $K=\prod_{v\notin\Sigma}K_v$, and a prime number $\ell >
[F:\Q]n+1$, we have $\ell \nmid |K_v|$ for infinitely many places $v$.
\end{lem}

\begin{proof}
The group $K_v$ embeds into a conjugate of
$\GL(n,R_{v})$. Therefore $|K_v|$ divides
$|\GL(n,R_{v})|=p^{\infty}\prod_{i=1}^n(q^i-1).$
Assume $\ell$ divides $|K_v|$ for almost all $v$. Then $p$ has
order at most $[F:\Q]n$ in $(\Z/\ell)^\times$ for almost all primes
$p$. Now, $(\Z/\ell)^\times$ is cyclic of order $\ell-1$, so by
Dirichlet's theorem on primes in arithmetic progressions we
conclude that $\ell \leq [F:\Q]n+1$. 
\end{proof}

\section{Ihara's Lemma}

\subsection{Parahoric Subgroups}

From now on we assume for simplicity that $G^{\der}$ is
simple (that is, it has no nontrivial connected normal subgroups).
Moreover, we fix a compact open subgroup
$K=\prod_{v \notin \Sigma}K_v \subset G(\A^{\Sigma}).$
Then $K_v \subset G_v$ is a hyperspecial maximal
compact subgroup for almost all places $v$, that is,
$K_v=\GG(R_v)$ for a smooth affine group
scheme $\GG$ of finite type over $R_v$ with generic fiber $G$. 
Such $\GG$ exists precisely when $G_v$ is
unramified. Let us look at a fixed finite place $w$ of $F$ where
$K_w$ is hyperspecial. Then write $K=K_wK^w$, where
$$K^w=\prod_{v \neq w}K_v \subset G(\A^{\Sigma,w}).$$
Let $\mathcal{B}_w$ denote the reduced Bruhat-Tits building of
$G_w$ (that is, the building of $G_w^{\ad}$). We have
assumed $G^{\der}$ is simple, so $\mathcal{B}_w$ is a
simplicial complex. Let $x \in \mathcal{B}_w$ be the vertex fixed
by $K_w$. Let $(x,x')$ be an edge in the building. 
Consider the maximal compact subgroup $K_w' \subset G_w$ fixing
the vertex $x'$, and the parahoric subgroup $J_w=K_w \cap K_w'$
associated with the edge $(x,x')$. Let $K'=K_w'K^w$ and $J=J_wK^w$
be the corresponding subgroups of $G(\A^{\Sigma})$.

\begin{lem}
$\la K_w,K_w' \ra=G_w^0:=\{g \in G_w: |\chi(g)|=1, \forall \chi
\in X^*(G)_{F_w}\}$.
\end{lem}

\begin{proof}
This follows from Bruhat-Tits theory. 
\end{proof}

Note that $G_w^{\der} \subset G_w^0 \subset G_w^1
=G_w \cap G(\A)^1$.

\subsection{The Concrete Setup}

We apply the general results when $L$ is $L_{\lambda_0}$ and 
$\OO$ is $R_{\lambda_0}$. Let $H=Z(\HH_{J,L_{\lambda_0}})$. This is a 
commutative $L_{\lambda_0}$-algebra. It comes with the involution defined 
by $\phi^{\vee}(x)=\phi(x^{-1})$. The $L_{\lambda_0}$-space 
$V=\aA(J,\rho,L_{\lambda_0})$ is 
finite-dimensional; $Z(\HH_{J,L_{\lambda_0}})$ acts on it.
The order $Z(\HH_{J,R_{\lambda_0}})$ preserves the lattice
$V_{R_{\lambda_0}}=\aA(J,\rho,R_{\lambda_0})$. The space $V$ comes
with the bilinear form $\la -,- \ra_J$. The compatibility
conditions between these data are satisfied. Let $U=\aA(K,\rho,L_{\lambda_0})
\oplus \aA(K',\rho,L_{\lambda_0})$. Then $Z(\HH_{J,L_{\lambda_0}})$ acts on this
space via the natural maps to $Z(\HH_{K,L_{\lambda_0}})$ and
$Z(\HH_{K',L_{\lambda_0}})$. The lattice $U_{R_{\lambda_0}}=
\aA(K,\rho,R_{\lambda_0}) \oplus \aA(K',\rho,R_{\lambda_0})$
is preserved by $Z(\HH_{J,R_{\lambda_0}})$. The bilinear
form on $U$ is given by the sum $\la -,- \ra_K \oplus \la -,-
\ra_{K'}$. The degeneracy map $\delta$ is given by
$$
\delta: \aA(K,\rho,L_{\lambda_0}) \oplus \aA(K',\rho,L_{\lambda_0})
\overset{\operatorname{sum}}{\to} \aA(J,\rho,L_{\lambda_0}),
$$
which is clearly $Z(\HH_{J,L_{\lambda_0}})$-linear. Obviously,
$\ker\delta$ consists of all pairs $(f,-f)$, where
$$
f \in \aA(K,\rho,L_{\lambda_0}) \cap\aA(K',\rho,L_{\lambda_0})
=\{\text{$G_w^0K^w$-invariant functions $f \in\aA$}\}.
$$
The decompositions $U=\ker\delta\oplus(\ker\delta)^{\bot}$ and
$V=\im\delta\oplus(\im\delta)^{\bot}$ are immediate because of
the relation between the pairings and the inner product.

\subsection{Combinatorial Ihara Lemma}

The proof of the following lemma is a straightforward
generalization of [T1, p. 274]. It asserts that the quotient
$V_{R_{\lambda_0}}\cap\delta(U)\supset\delta(U_{R_{\lambda_0}})$
is killed by $C=1$.

\begin{lem}
We have that $\aA(J,\rho,R_{\lambda_0}) \cap 
\delta[\aA(K,\rho,L_{\lambda_0}) \oplus
\aA(K',\rho,L_{\lambda_0})]$ is equal to
$\delta[\aA(K,\rho,R_{\lambda_0}) \oplus
\aA(K',\rho,R_{\lambda_0})]$.
\end{lem}

\begin{proof}
Let us first set up some machinery for the proof. There are natural
projections $\pi:X_J=G(F)\bs G(\A^\infty)/J'{}'\to X_K$ and $\pi:X_J\to X_{K'}$.
We define an equivalence relation on $X_J$ by saying that $x,y \in
X_J$ are equivalent ($x \sim y$) iff there exists a chain $x=x_0,\ldots,x_d=y$ 
such that $\forall i$: $\pi(x_i)=\pi(x_{i+1})$ or $\pi'(x_i)=\pi'(x_{i+1})$.
This gives a partition of $X_J$ into equivalence classes $X_J^j$.
For each $j$, we fix a representative $y^j \in X_J^j$.
Correspondingly, we have a radius function $d:X_J \to
\Z_{\geq 0}$ defined as follows. Given $x \in X_J$, there is a
unique $j$ such that $x \sim y^j$. Then $d(x)$ is the minimal
length of a chain connecting $x$ to $y^j$. Now, suppose
$g=\delta(f,f') \in \aA(J,\rho,R_{\lambda_0})$ for some $f \in
\aA(K,\rho,L_{\lambda_0})$ and $f' \in \aA(K',\rho,L_{\lambda_0})$. 
We want to show $g\in \delta(\aA(K,\rho,R_{\lambda_0}) \oplus
\aA(K',\rho,R_{\lambda_0}))$.

We claim that we may assume that $f(\pi(y^j))=0$ for all $j$.
To see this, note that $X_K=\sqcup \pi(X_J^j)$ and 
$X_{K'}=\sqcup\pi'(X_J^j)$. We then define 
$\wt{f} \in \aA(K,\rho,L_{\lambda_0})$ such
that $\wt{f}|\pi(X_J^j)\equiv f(\pi(y^j))$, and 
$\wt{f}' \in\aA(K',\rho,L_{\lambda_0})$ such that 
$\wt{f}'|\pi'(X_J^j)\equiv f(\pi(y^j))$. Then
$$
g=\delta(f-\wt{f},f'+\wt{f}')
$$
and $(f-\wt{f})(\pi(y^j))=0$ for all $j$. This proves the claim.

From now on assume that $f(\pi(y^j))=0$ for all $j$. 
We claim, for every $m\ge 0$, that for every $x \in X_J$ with $d(x)=m$
we have that $f(\pi(x)) \in R_{\lambda_0}$ and $f'(\pi'(x)) \in R_{\lambda_0}$.
We prove this by induction on $m \geq 0$. The case $m=0$ is 
essentially just our assumption. Assume the statement is true for 
$m-1 \geq 0$. Consider $x \in X_J$ with $d(x)=m$. Let 
$x=x_0,x_1,\ldots,x_m=y^j$ be a chain of minimal length. 
Then $x'=x_1 \in X_J$ has $d(x')=m-1$, so by induction 
$f(\pi(x')) \in R_{\lambda_0}$ and $f'(\pi'(x')) \in R_{\lambda_0}$. 
However,  $\pi(x)=\pi(x')$ or $\pi'(x)=\pi'(x')$. In either case we get the 
statement for $x$. This proves the lemma, for then 
$f \in \aA(K,\rho,R_{\lambda_0})$ and 
$f' \in \aA(K',\rho,R_{\lambda_0})$. Note that $f(\pi(x))
\in R_{\lambda_0}$ if and only if $f'(\pi'(x)) \in R_{\lambda_0}$. 
\end{proof}

\section{Applying the Abstract Theory}

\subsection{Computing $\delta^{\vee}\delta$}

To apply the abstract theory it is necessary to compute
$\delta^{\vee}\delta$ explicitly.

\begin{lem} The endomorphism $\delta^{\vee}\delta$ is given by the
$2\times 2$ matrix
$$
\delta^{\vee}\delta=\begin{pmatrix}[K:J] & [K:J]e_K \\
[K':J]e_{K'} & [K':J] \end{pmatrix}.
$$
\end{lem}

\begin{proof}
Put $k=[K:J]$ and $k'=[K':J]$. Recall that 
$\delta\begin{pmatrix}\phi_1^K\\ \phi_2^{K'}\end{pmatrix}
=\phi_1^K+\phi_2^{K'}\in\aA(K',\rho,L_{\lambda_0})$. Also,
$\delta^\vee\phi^J=\begin{pmatrix}ke_K\phi^J+ke_Ke_{K'}\phi^J\\
k'e_{K'}e_K\phi^J+k'e_{K'}\phi^J\end{pmatrix}$.

Let us write the endomorphism $\delta^{\vee}\delta$ of
$\begin{pmatrix}\aA(K,\rho,L_{\lambda_0}) \\ 
\aA(K',\rho,L_{\lambda_0})\end{pmatrix}$ as
$\delta^{\vee}\delta=\begin{pmatrix}a & b \\ c & d\end{pmatrix},$
where $b: \aA(K',\rho,L_{\lambda_0}) \to \aA(K,\rho,L_{\lambda_0})$ 
and so on. Then  
$$\la f_1+f_2,g_1+g_2\ra_J=\bigg\la\begin{pmatrix}f_1\\ f_2\end{pmatrix},
\begin{pmatrix}a & b \\ c & d\end{pmatrix}
\begin{pmatrix}g_1\\ g_2\end{pmatrix}\bigg\ra_J$$
$$=\bigg\la\begin{pmatrix}f_1\\ f_2\end{pmatrix},
\begin{pmatrix}ag_1+bg_2 \\  cg_1+dg_2\ra_J \end{pmatrix}\bigg\ra_J
=\la f_1,ag_1+bg_2\ra_K+\la f_2,cg_1+dg_2\ra_{K'}.$$
Taking $g_2=0=f_2$ we get $ag_1=kg_1$; with $g_1=0=f_2$, we get
$bg_2=ke_Kg_2$; with $f_1=0=g_2$ we get $cg_1=k'e_{K'}g_1$;
with $f_1=0=g_1$ we get $dg_2=k'g_2$.
\end{proof}

\subsection{The Main Lemma}

In our situation, Corollary 2.1 gives the following crucial lemma.

\begin{lem}
Let $f \in \aA(K,\rho,R_{\lambda_0})$ be an eigenform for
$Z(\HH_{K,R_{\lambda_0}})$ with character
$\eta_f:\T_{K,R_{\lambda_0}} \to R_{\lambda_0}$. Let $\ell$ be the
residual characteristic of $\lambda_0$. Suppose there exist at least two 
places $v$ such that $\ell \nmid |K_v|$. Assume that there exists
a place $w$ of $F$ such that $f$ modulo $\lambda_0$ is not 
$G_w^0$-invariant. Consider
$e_{K,K'}=[K:J][K':J]_K(e_K \ast e_{K'}\ast e_K) \in Z(\HH_{K,\Z})$
where $[K':J]_K=k'/(k',k)$ in the notation of the proof of Lemma 8.1.
Suppose there is $n$ with
$$0<n\le v_{\lambda_0}(\eta_f(e_{K,K'})-[K:J][K':J]_K)-v_{\lambda_0}([K':J]_K).$$
Then the reduction of $\eta_f\circ\ast e_K$ modulo $\lambda_0^n$ 
factors through $\T_{J,R_{\lambda_0}}^{\new}$. 
\end{lem}

\begin{proof}
(1) First we produce an eigenvector for $Z(\HH_{J,R_{\lambda_0}})$
in\hb $U_{R_{\lambda_0}}$ $=\aA(K,\rho,R_{\lambda_0}) \oplus 
\aA(K',\rho,R_{\lambda_0})$. For that we take
$$
\vec{f}=[K':J]_K(f,-r(e_{K'})f) \in \aA(K,\rho,R_{\lambda_0})
\oplus \aA(K',\rho,R_{\lambda_0}).
$$
The factor $[K':J]_K$ is included since $r(e_{K'})f$ does not
necessarily take values in $R_{\lambda_0}$: note that $e_{K'}
=\chi_{K'}/\mu(K')=k\chi_{K'}/k'\mu(K)$, thus 
$k'e_{K'}/(k',k)\in\HH_{K,\Z}$. Clearly, $\vec{f}$ 
is an eigenvector for $Z(\HH_{J,R_{\lambda_0}})$. Its
character is the composite
$$
\eta_{\vec{f}}: Z(\HH_{J,R_{\lambda_0}}) \overset{\ast e_K} \to
Z(\HH_{K,R_{\lambda_0}}) \overset{\eta_f}{\to}R_{\lambda_0}.
$$
Indeed, for $h\in Z(\HH_{J,R_{\lambda_0}})$, $f\in \aA(K,\rho,R_{\lambda_0})$,
we have $h(f,-e_{K'}\cdot f)=(he_Kf,-e_{K'}he_Kf)=\eta_f(he_K)(f,-e_{K'}f)$.

\n (2) Using the explicit formula for $\delta^{\vee}\delta$ in lemma 8.1
above, it follows that
$$
\delta^{\vee}\delta(\vec{f})=m(-f,0),\qquad m=\eta_f(e_{K,K'})-[K:J][K':J]_K.
$$
Note that $(-f,0) \in U_{R_{\lambda_0}}$. We claim that Corollary 2.1 applies with
this $m\in R_{\lambda_0}$. Indeed, $m\not=0$. If $m=0$ then $\vec{f}$ must belong 
to the kernel of $\delta$. Then $f$ must be invariant under the group $G_w^0$ 
(say, on the right), contradicting our assumption. 

\n (3) Define
$$
\FF=\{x \in L: xf \in\aA(K,\rho,R_{\lambda_0})+
\aA(K,\rho,L_{\lambda_0})\cap\aA(K',\rho,L_{\lambda_0})\}.
$$
This is an $R_{\lambda_0}$-submodule of $L$ containing $R_{\lambda_0}$. 
We have $\FF=L$ if $f \in\aA(K,\rho,L_{\lambda_0})\cap\aA(K',\rho,L_{\lambda_0})$. 
However, $f$ is not $G_w^0$-invariant, hence 
$f \notin\aA(K,\rho,L_{\lambda_0})\cap\aA(K',\rho,L_{\lambda_0})$. 

We claim that $\FF$ is a fractional ideal. To see this note that if
$x\in\FF$, $xf \in\aA(K,\rho,R_{\lambda_0})+
\aA(K,\rho,L_{\lambda_0})\cap\aA(K',\rho,L_{\lambda_0})$. Then 
$A_K\la xf,g\ra_K\in R_{\lambda_0}$ by Lemma 6.2. Hence
$A_K \la f,g \ra_K \FF \subset R_{\lambda_0}$
for every $g \in \aA(K,\rho,R_{\lambda_0}) \cap
(\aA(K,\rho,L_{\lambda_0})\cap\aA(K',\rho,L_{\lambda_0}))^{\bot}$. 
These $g$ span
$(\aA(K,\rho,L_{\lambda_0})\cap\aA(K',\rho,L_{\lambda_0}))^{\bot}$ 
so $f$ must belong to
$\aA(K,\rho,L_{\lambda_0})\cap\aA(K',\rho,L_{\lambda_0})$ 
if $\la f,g \ra_K=0$ for all such $g$. Thus $\la f,g\ra_K$ is not identically
zero, and $\FF$ is a fractional ideal.

\n (4) Now, the nonzero ideal $\wt{\EE}=\FF^{-1}$ satisfies:
$$
\wt{\EE}(Lf \cap[\aA(K,\rho,R_{\lambda_0})
+\aA(K,\rho,L_{\lambda_0})\cap\aA(K',\rho,L_{\lambda_0})])
\subset R_{\lambda_0}f.
$$
Therefore, $\EE=[K':J]_K\wt{\EE}$ satisfies the primitivity 
condition in corollary 2.1 (recall that $\vec{f}=[K':J]_K(f,-r(e_{K'})f)$):
$$
\EE(L\vec{f} \cap (\aA(K,\rho,R_{\lambda_0}) \oplus
\aA(K',\rho,R_{\lambda_0})+\ker \delta)) \subset
R_{\lambda_0}\vec{f}.
$$

Suppose that $v_{\lambda_0}(\wt{\EE})\neq 0$. Then 
$\FF^{-1}\subset\lambda_0$, thus $\lambda_0^{-1}\subset \FF$. 
It follows that $f \in\lambda_0(\aA(K,\rho,R_{\lambda_0})
+\aA(K,\rho,L_{\lambda_0})\cap\aA(K',\rho,L_{\lambda_0}))$.
Hence $f=\lambda_0(g+h)$, thus $f-\lambda_0g
=\lambda_0h\in\aA(K,\rho,L_{\lambda_0})
\cap\aA(K',\rho,L_{\lambda_0})$ is $G_w^0$-invariant.
Hence the reduction $\bar{f} \in \aA(K,\rho,\F_{\lambda_0})$
is $G_w^0$-invariant. Hence $v_{\lambda_0}(\wt{\EE})=0$.

Since $\ell \nmid |K_v|$ holds for at least one $v \neq w$, by 
assumption, we can find $A_K$ and $A_{K'}$ indivisible by $\ell$ 
according to Lemma 6.2. Also we can take $C=1$ by Lemma 7.2. So if 
$0<n\le v_{\lambda_0}(m)- v_{\lambda_0}([K':J]_K)\le v_{\lambda_0}
(mE^{-1}\EE^{-1})$, by Lemma 2.1, $\eta_{\vec{f}}$mod$\lambda_0^n$
factors through $\T_{J,R_{\lambda_0}}^{\new}\to R_{\lambda_0}/\lambda_0^n$.
\end{proof}

\section{Semisimplicity}

\subsection{Semisimplicity in Characteristic Zero}

Let $\pi$ be an automorphic representation of $G(\A)$ with $\pi_\infty=\rho_\infty$, 
$\pi_\Sigma\supset\rho_\Sigma$ and nonzero space $\pi^K$ of $K$ ($\subset
G(\A^\Sigma)$)-fixed vectors.
It is known that each $\pi^K$ is a simple module over $\HH_K$. 
Hence $\aA(K,\rho,\C)$ is semisimple.
Moreover, by Schur's lemma, the center $Z(\HH_K)$ acts on
$\pi^K$ by a $\C$-algebra homomorphism
$\eta_{\pi^K}:Z(\HH_K) \to \C$. For a character
$\eta:Z(\HH_K) \to \C$, we denote by
$\aA(K,\rho,\C)(\eta)$ the $\eta$-isotypic component. That is, the
eigenspace
$\aA(K,\rho,\C)(\eta)=\{f \in \aA(K,\rho,\C); r(\phi)f=\eta(\phi)f,
\forall \phi \in Z(\HH_K)\}.$
Then there is a direct sum decomposition
$\aA(K,\rho,\C)=\bigoplus_{\eta}\aA(K,\rho,\C)(\eta)$. Clearly,
$\aA(K,\rho,\C)(\eta) \neq 0$ if and only if $\eta=\eta_{\pi^K}$
for some $\pi$. The image $\T_K \subset \End_\C \aA(K,\rho,\C)$ 
of the center $Z(\HH_K)$ is a commutative semisimple
$\C$-algebra, that is, a direct product of copies of $\C$.

\begin{lem}
The eigenspace $\aA(K,\rho,\C)(\eta)$ is nonzero if and only if
$\eta$ factors through $\T_K$.
\end{lem}

\begin{proof}
Obviously, $\eta$ factors if $\aA(K,\rho,\C)(\eta)\neq 0$.
Conversely, suppose $\eta$ factors and look at its kernel
$\m=\ker(\eta) \subset \T_K$. This is a maximal ideal since $\im(\eta)=\C$ 
is a field.  Since $\T_K$ acts faithfully on $\aA(K,\rho,\C)$, which is
finite-dimensional, $\m$ belongs to the support of $\aA(K,\rho,\C)$,
namely the localization $\aA(K,\rho,\C)_\m$ is nonzero.
By the theory of associated primes, $\m$ contains a prime ideal of the form 
$\Ann_{\T_K}(f)$ with $f \in\aA(K,\rho,\C)$ 
(Dummit and Foote, 3rd Ed., Sect. 15.4, Ex. 40, p. 730). All primes are 
maximal in $\T_K$, so in fact $\m=\Ann_{\T_K}(f)$. Clearly $\m$ contains 
$T-\eta(T)$ for every $T \in \T_K$, so $f \in \aA(K,\rho,\C)(\eta)$, 
and $f$ must be nonzero as $\m \neq \T_K$. 
\end{proof}

Now, consider the $\HH_{K,\Q}$-module $\aA(K,\rho,L)$, and 
the image $\T_{K,\Q}$ of the center $Z(\HH_{K,\Q})$ in the endomorphism 
algebra $\End_{L}\aA(K,\rho,L)$. The algebra $\T_{K,\Q}$ can be viewed 
as a subring of $\T_K \simeq \C \otimes_{\Q} \T_{K,\Q}$. We deduce that
$\T_{K,\Q}$ is a reduced commutative finite-dimensional
$\Q$-algebra, that is, a product of number fields by Nakayama's
lemma:
$$
\T_{K,\Q} \simeq L_1 \times \cdots \times L_t.
$$
Visibly, $\T_{K,\Q}$ is a semisimple $\Q$-algebra. (The $L_i$
occurring in $\T_{K,\Q}$ are totally real or CM.)

\subsection{Semisimplicity in Positive Characteristic}

Now let $R$ be a field of characteristic $p>0$. We are interested
in when $\aA(K,\rho,R)$ is a semisimple module over
$Z(\HH_{K,R})$. As we have just seen, this means that $\T_{K,R}$ 
is a semisimple $R$-algebra. We have 
$\T_{K,R}\simeq R\otimes_{\F_p}\T_{K,\F_p}$, so equivalently, 
when is $\T_{K,\F_p}$ semisimple? 

There is always a surjective homomorphism
$\xi:\F_p \otimes_{\Z} \T_{K,\Z} \onto \T_{K,\F_p}.$
Indeed the image of $\F_p \otimes_{\Z} \T_{K,\Z}$ in
$\End_{\F_p}\aA(K,\rho,\F_p)$ equals the image of $\F_p
\otimes_{\Z} Z(\HH_{K,\Z})$, and the natural map from
$\F_p \otimes_{\Z} \T_{K,\Z}$ to $Z(\HH_{K,\F_p})$ is onto. 

Put
$\wt{\T}_{K,\Z}=\{T \in \T_{K,\Q}; T(\aA(K,\rho,\Z))\subset
\aA(K,\rho,\Z)\}.$
This is a free finite $\Z$-module containing $\T_{K,\Z}$ as a
subgroup of finite index.

\begin{lem}
The kernel $\ker\xi$ is nilpotent. It is trivial iff
$p \nmid [\wt{\T}_{K,\Z}:\T_{K,\Z}]$.
\end{lem}

\begin{proof}
For the first assertion it is enough to show that every element in 
$\ker(\xi)$ is nilpotent. Under the identification 
$\F_p \otimes_{\Z} \T_{K,\Z}\simeq \T_{K,\Z}/p\T_{K,\Z}$, the kernel 
$\ker(\xi)$ corresponds to the ideal
$(\T_{K,\Z} \cap p\wt{\T}_{K,\Z})/p\T_{K,\Z}.$
Let $T \in \T_{K,\Z} \cap p\wt{\T}_{K,\Z}$. Obviously,
$\wt{\T}_{K,\Z}$ is integral over $\Z$, so there is an equation
$$
(p^{-1}T)^n+a_{n-1}(p^{-1}T)^{n-1}+\cdots+a_1(p^{-1}T)+a_0=0
$$
for certain $a_i \in \Z$. Multiplying by $p^n$ we see that $T^n\in p\T_{K,\Z}$. 

For the last assertion, note that $\ker\xi=0$
if and only if $\F_p \otimes_{\Z} \T_{K,\Z} \to \F_p
\otimes_{\Z} \wt{\T}_{K,\Z}$ is injective. 
\end{proof}

In particular, $\ker\xi$ is contained in the Jacobson radical.
We let $\bar{\T}_{K,\Z}$ denote the integral closure of $\Z$ in
$\T_{K,\Q}$. It contains $\wt{\T}_{K,\Z}$ as a subgroup of
finite index.

\begin{lem}
If $p \nmid \Delta_K:=
[\bar{\T}_{K,\Z}:\wt{\T}_{K,\Z}]\cdot\prod_{i}\Delta_{L_i/\Q}$
then $\T_{K,\F_p}$ is semisimple.
\end{lem}

\begin{proof}
Note that $\F_p \otimes_{\Z}\wt{\T}_{K,\Z}
\simeq \F_p \otimes_{\Z}\bar{\T}_{K,\Z}$ since $p \nmid
[\bar{\T}_{K,\Z}:\wt{\T}_{K,\Z}]$. 

As $p$ does not divide the
discriminant $\Delta_{L_i/\Q}$ for each $i$, $p$ is unramified in 
every $L_i$ occurring in $\T_{K,\Q}$. Hence
$\F_p \otimes_{\Z}\bar{\T}_{K,\Z} \simeq \prod_i
R_{L_i}/pR_{L_i} \simeq \prod_i\prod_{\p|p}R_{L_i}/\p.$

There is an embedding
$\T_{K,\F_p} \simeq \T_{K,\Z}/\T_{K,\Z} \cap p\wt{\T}_{K,\Z}
\hookrightarrow \wt{\T}_{K,\Z}/p\wt{\T}_{K,\Z} \simeq \F_p
\otimes_{\Z}\wt{\T}_{K,\Z}.$
It follows that $\T_{K,\F_p}$ is semisimple. 
\end{proof}

The converse holds at least for $p \nmid
[\wt{\T}_{K,\Z}:\T_{K,\Z}]$ (that is, when $\xi$ is injective).

\subsection{The Simple Modules}

Let $R$ be a perfect field of characteristic $p \geq 0$. Up to
isomorphism, a simple $Z(\HH_{K,R})$-module is given
by an extension $R'/R$ with an action given by a surjective
$R$-algebra homomorphism $\eta:Z(\HH_{K,R})
\onto R'$. If $(\eta,\,R')$ is such a submodule of
$\aA(K,\rho,R)$, the extension $R'/R$ is finite and $\eta$
factors through $\T_{K,R}$. If $p \nmid \Delta_K$, there exists a
finite extension $L/R$ such that we have a direct sum
decomposition
$$
\aA(K,\rho,L)=\bigoplus_{\eta}\aA(K,\rho,L)(\eta).
$$
This is still true when $p|\Delta_K$, provided
$\aA(K,\rho,L)(\eta)$ denotes the generalized eigenspace:
$$
\{\text{$f \in \aA(K,\rho,L)$;
$\forall \phi \in  Z(\HH_{K,L})$,
$(r(\phi)-\eta(\phi))^nf=0$  $\exists n \geq 1$}\}.
$$
Observe the following:

\begin{lem} Let $R$ be a field. Choose a finite extension
$L/R$ as above. Let $L'/L$ be an arbitrary extension. Suppose
$\eta': Z(\HH_{K,L'}) \to L'$ occurs in
$\aA(K,\rho,L')$. Then $\eta'=1\otimes \eta$ for some
character $\eta: Z(\HH_{K,L}) \to L$ occurring in
$\aA(K,\rho,L)$. Moreover,
$$
\aA(K,\rho,L')(1\otimes \eta) \simeq L' \otimes_L
\aA(K,\rho,L)(\eta),
$$
so $\eta$ and $\eta'=1 \otimes \eta$ occur with the same multiplicity.
\end{lem}

\begin{proof}
Both $\aA(K,\rho,L)$ and $\aA(K,\rho,L') \simeq
L' \otimes_L \aA(K,\rho,L)$ have decompositions into direct
sums of generalized eigenspaces. Under this isomorphism, $L'
\otimes_L \aA(K,\rho,L)(\eta) \hookrightarrow
\aA(K,\rho,L')(1\otimes \eta)$. Therefore, every $\eta'$
occurring in $\aA(K,\rho,L')$ must come from an $\eta$, and
the above injection must be an isomorphism. 
\end{proof}

Let us apply these results to a number field $R=L'$. We conclude that there
exists a number field $L/L'$ such that $\aA(K,\rho,L)$ is a
direct sum of eigenspaces for characters $Z(\HH_{K,L})
\to L$. Furthermore, if $\eta: Z(\HH_K)\to \C$ 
is a character such that $\aA(K,\rho,\C)(\eta)\neq 0$, 
then $\eta$ restricts to a $\Q$-algebra homomorphism
$Z(\HH_{K,\Q}) \to L$ occurring in
$\aA(K,\rho,L)$. In addition, since $Z(\HH_{K,\Z})$
preserves $\aA(K,\rho,R_L)$, $\eta$ even restricts
to a ring homomorphism $Z(\HH_{K,\Z}) \to
R_L$ occurring in $\aA(K,\rho,R_L)$.

\section{End of Proof}

\subsection{Invariance Modulo $\lambda$}
Denote by $\aA^0(K,\rho,\F_\lambda)$ the
space of the nonabelian modulo $\lambda$ relative to $K$ automorphic 
forms in $\aA(K,\rho,\F_\lambda)$.

\begin{lem}
Choose a number field $L/\Q$ such that $\aA(K,\rho,L)$ is a 
direct sum of eigenspaces. Put $R=R_L$. Let $\pi$ be an 
automorphic representation of $G(\A)$ such that $\pi^K\not=0$, 
$\pi_\Sigma\supset\rho_\Sigma$,
and $\pi_{\infty}=\rho_\infty$. Denote by
$\eta=\eta_{\pi^K}: Z(\HH_{K,\Z}) \to R$ the character 
giving the action on $\pi^K$. Let $w$ be a place such that $K_w$ is 
hyperspecial, thus $G_w$ is unramified. Suppose $\pi$ is 
${\operatorname{non}}$ abelian modulo $\lambda$ relative to $K$, and 
$\ov{\eta}:Z(\HH_{K,\Z})\to\F_{\lambda_0}=R_{\lambda_0}/\lambda_0$ 
denotes the reduction of $\eta$. Then the eigenspace 
$\aA^0(K,\rho,\F_\lambda)(\bar{\eta})$ 
contains no nonzero $G_w^{\der}$-invariant functions.
\end{lem}

\begin{proof}
As observed above, $\eta$ occurs in $\aA(K,\rho,R_{\lambda_0})$, 
that is, there exists an eigenform $0 \neq f \in \aA(K,\rho,R_{\lambda_0})$
with $\eta_f=\eta$. Let $\bar{f}=1 \otimes f \in \aA(K,\rho,\F)$ be
the reduction of $f$ modulo $\lambda$, where $\F=R_\lambda/\lambda$ 
is a finite extension of $\F_{\ell}$. By scaling $f$, we can assume that 
$\bar{f} \neq 0$. Let us assume $\bar{f}$ is 
$G_w^{\der}$-invariant. Now, $G^{\der}$ 
is simple, simply connected and $G_w^{\der}$ is noncompact. 
By the strong approximation theorem, $\bar{f}$ is in fact
$G^{\der}(\A^{\infty})$-invariant. In particular, $\dim\rho_\Sigma=1$.
As $H^1(F_v,G^{\der})=0$ for each finite place $v$, there is a short exact sequence
$$
1 \to G^{\der}(\A^{\infty}) \to G(\A^{\infty})
\overset{\nu}{\to}G^{\ab}(\A^{\infty}) \to 1.
$$
It follows that $\bar{f}$ lives on $G^{\ab}(\A^{\infty})$.
More precisely, there exists a unique function
$\wt{f}:G^{\ab}(\A^{\infty}) \to \F$ such that
$\bar{f}=\wt{f} \circ \nu$. It fits into the diagram
\begin{align*}
\xymatrix{ X_K=G(F) \backslash G(\A^{\infty})/K'{}' \ar[r]^-{\bar{f}}
\ar@{->>}[d]_{\nu} &
\F \\
Y_K=\nu(G(F))\backslash G^{\ab}(\A^{\infty})/\nu(K'{}')
\ar[ur]_-{\wt{f}} }
\end{align*}

If $R$ is a ring we denote by $\aA(K,\rho,R)^{\ab}$ 
the module of $R$-valued functions on $Y_K$. Pulling back via $\nu$,
identifies $\aA(K,\rho,R)^{\ab}$ with an
$\HH_{K,R}$-submodule of $\aA(K,\rho,R)$. Then $0 \neq
\wt{f} \in \aA(K,\rho,\F)^{\ab,0}(\bar{\eta})$.
By Lemme 6.11 of [DS, p. 522] we can lift $\bar{\eta}$ to 
characteristic zero: there exists an eigenform $0 \neq f' \in
\aA(K,\rho,L_{\lambda})^{\ab}$ such that its 
character $\eta': Z(\HH_{K,\Z}) \to R_{\lambda}$
reduces to $\bar{\eta}$ modulo $\lambda$. From
the results of the previous section we see that in fact $\eta'$
maps into $R$, and it occurs in
$\aA(K,\rho,L)^{\ab}$ (and therefore in
$\aA(K,\rho,L_{\lambda_0})^{\ab}$). However,
$\aA(K,\rho,L_{\lambda_0})^{\ab}$ is just the 
space of $L_{\lambda_0}$-valued functions on the finite abelian group 
$Y_K$, so the characters form a basis. We conclude that there exists a 
character $\chi$ such that $\eta(\phi) \equiv \eta_{\chi}(\phi)$ (mod 
$\lambda$) for all $\phi \in Z(\HH_{K,\Z})$. This contradicts the
assumption that $\pi$ is nonabelian mod $\lambda$ relative to $K$.
Hence $\aA^0(K,\rho,\F_\lambda)(\bar{\eta})$ 
contains no nonzero $G_w^{\der}$-invariant functions.
\end{proof}

\subsection{Proof of Theorem 0.3}
Note that $\pi\subset\aA(K,\rho,L)$ for some number field $L$.
The reduction $\bar{\eta}_{\pi^K}$ modulo $\lambda \cap R_L$ factors 
through $\T_{J,R_{\lambda_0}}^{\new}$, where $R_{\lambda_0}$ is the 
completion of $R_L$ at $\lambda_0$, by the main lemma (Lemma 8.2). 
That is, there exists a character
$\eta':Z(\HH_{J,R_{\lambda_0}}) \to \F_\lambda$ 
factoring through $\T_{J,R_{\lambda_0}}^{\new}$ such that 
$\eta'(\phi)=\eta_{\pi^K}(e_K\ast \phi)$ (mod $\lambda$) for all 
$\phi \in Z(\HH_{J,R_{\lambda_0}})$. As above, there is 
a surjective homomorphism with nilpotent kernel
$$
\F_{\lambda_0} \otimes_{R_{\lambda_0}} \T_{J,R_{\lambda_0}}^{\new} 
\onto\T_{J,\F_{\lambda_0}}^{\new}.
$$
Thus $\eta'$ gives rise to a character $\T_{J,\F_{\lambda_0}}^{\new}
\to \F_{\lambda_0}$, also denoted by $\eta'$. By a standard 
argument (used above in section 8.2), there is an eigenform $f' \in
\aA(J,\rho,\F_{\lambda_0})^{\new}$ with character $\eta'$. 
Now we apply the Deligne-Serre lifting lemma, [DS, p. 522], to the finite
free module $\aA(J,\rho, R_{\lambda_0})^{\new}$. It gives the 
existence of a character $\wt{\eta}: \T_{J,R_{\lambda_0}}^{\new} 
\to\wt{R}_{\lambda_0}$ occurring in
$\aA(J,\rho,\wt{R}_{\lambda_0})^{\new}$ and
reducing to $\eta'$, where $\wt{R}_{\lambda_0}$ is the
ring of integers in a finite extension of $L_{\lambda_0}$. Since
$\T_{J,R_{\lambda_0}}^{\new}$ preserves the lattice
$\aA(J,\rho,R_{\lambda_0})^{\new}$, the values 
$\wt{\eta}(\phi)$ all lie in the ring of integers of some 
number field, $R_{\wt{L}}$. We deduce that there exists 
a character $\wt{\eta}: Z(\HH_{J,R_L}) \to
R_{\wt{L}}$, occurring in
$\aA(J,\rho,L)^{\new}$, such that
$$\wt{\eta}(\phi) \equiv \eta_{\pi^K}(e_K \ast \phi)(\md\lambda)$$
for all $\phi \in Z(\HH_{J,R})$. From the decomposition
of $\aA(J,\rho,L_{\lambda_0})$ in terms of automorphic 
representations, it follows that the new space 
$\aA(J,\rho,L_{\lambda_0})^{\new}$ has the following description:
$$
\aA(J,\rho,L_{\lambda_0})^{\new} \simeq {\bigoplus}_{\{\pi \in
\Pi_{\operatorname{unit}}(G(\A)) ; \pi_{\infty}
=\rho_\infty,\,\pi_\Sigma\supset\rho_\Sigma\}}m(\pi)\cdot\pi^J/(\pi^K+\pi^{K'}),
$$
as $Z(\HH_J)$-modules. The center $Z(\HH_J)$ acts
on the quotient $\pi^J/(\pi^K+\pi^{K'})$ by the character
$\eta_{\pi^J}$. We conclude that there exists an automorphic
representation $\wt{\pi}$ of $G(\A)$ with
$\wt{\pi}_{\infty}=\rho_\infty$, $\wt{\pi}_{\Sigma}\supset\rho_\Sigma$ 
and $\wt{\pi}^J\neq
\wt{\pi}^K+\wt{\pi}^{K'}$, such that
$\eta_{\wt{\pi}^J}=\wt{\eta}$. In particular,
$$\eta_{\wt{\pi}^J}(\phi) \equiv \eta_{\pi^K}(e_K \ast\phi)(\md\lambda)$$
for all $\phi \in Z(\HH_{J,R_L})$. This finishes the proof.\hfill $\square$

\section{Applications in Rank One}

\subsection{$\U(3)$ - the Nonsplit Case}

Let $F$ be a local nonarchimedean field. Suppose that the $F$-rank of 
$G^{\der}$ is one. In this rank one situation the parahoric
$J=K\cap K'$ is an Iwahori subgroup, denoted $I$.

\begin{lem}
If $\pi^I \neq\pi^K+\pi^{K'}$ then $\pi^K=\{0\}=\pi^{K'}$.
\end{lem}

\begin{proof}
Suppose $\pi^I\not=\{0\}$. Then $\pi$ is a constituent of a fully induced 
representation $\Ind(\chi)$, $\chi$ being an unramified character of the 
maximal torus $A$ in the Borel subgroup $B$ of $G$, by [Bo] or [B]. There 
are two cases. 

If $\pi=\Ind(\chi)$ then $\dim_{\C}\pi^K=1=\dim_{\C}\pi^{K'}$. 
Indeed, the building of $G$ is a tree and all vertices are special. Thus 
the maximal compact subgroup $K'$ is special, so we have the 
Iwasawa decomposition $G=BK'$, and $B\cap K'=A(R)$ is the maximal 
compact subgroup in the maximal torus $A$ in $B$. Then $f(g)=\chi(b)$, 
$g=bk'$, $b\in B$, $k'\in K'$, is well-defined, nonzero, fixed by $K'$. 

Now $\dim_{\C}\pi^I=[W]$, and the number of elements $[W]$ in the Weyl 
group $W$ of $A$ in $G$ is $2^{\rk(G^{\der})}$, namely 2. Our assumption 
is that $\pi^K+\pi^{K'}$ is not $\pi^I$, thus $\dim_\C(\pi^K+\pi^{K'})$ 
is 1. Hence $\pi^K=\pi^{K'}$ is a one-dimensional space fixed by $K$ 
and $K'$, hence by $G^0$ by Lemma 7.1, so that $\pi$ is a character, 
contradicting our assumption that $\pi=\Ind(\chi)$. 

The second case is where $\pi$ is strictly contained in $\Ind(\chi)$. 
By [Bo] or [B], each constituent of $\Ind(\chi)$ has an Iwahori invariant 
vector. Hence $\dim_\C\pi^I=1$. But $\pi^K+\pi^{K'}$ is strictly contained 
in $\pi^I$. Hence $\pi^K=\pi^{K'}=\{0\}$.
\end{proof}

\subsection{Proof of Theorem 0.4} This follows at once from Theorem 0.3,
using Lemma 11.1. 

In the case of $G=\U(3)$ where $w$ stays prime in $E$, let $\wt{\pi}$ be 
the automorphic representation we get from Theorem 0.3. By Lemma 11.1 and 
[Bo] or [B], 
$\wt{\pi}_w$ is a ramified constituent of a reducible unramified induced
representation. The constituents of the reducible unramified induced
representations are the nontempered one-dimensional and $\pi^\times$,
which are unramified, and the square integrable Steinberg and $\pi^+$.
See 11.3 below. But $\pi^{+,K'}\not=0$, hence $\wt{\pi}_w$ is Steinberg.

Finally, $[K_w:I_w]=q_w^3+1$, since $K_w$ is the fixer of a 
hyperspecial vertex $v$ in the Bruhat-Tits building, which has 
$q_w^3+1$ neighbors, and $I_w$ is the fixer of an edge $vv'$.
Thus $[K_w:I_w]$ counts the number of edges initiating from the
vertex $v$. Similarly $[K'_w:I_w]=q_w+1$ as $K'_w$ is the fixer 
of the special nonhyperspecial vertex $v'$, which has $q_w+1$
neighbors. As $q_w+1$ divides $q_w^3+1$, $[K'_w:I_w]_{K_w}=1$.
\hfill $\square$

\subsection{Reducibility of unramified representations}
Let $G$ be an unramified (split, or quasisplit and split over an
unramified extension $E$) reductive group over a $p$-adic field $F$.
An irreducible representation of $G$ has a nonzero vector fixed by
an Iwahori subgroup iff it is a constituent of a representation induced
from an unramified character of a minimal parabolic subgroup ([Bo] or
[B]). This induced representation is parametrized by the conjugacy 
class of a semisimple element $s$ in the connected dual group $\widehat{G}$
if $G$ is split, and in $\widehat{G}\sigma$ if $G$ is quasisplit and
splits over an unramified extension $E/F$, which we take to be minimal,
and denote by $\sigma$ a generator of the cyclic group $\Gal(E/F)$.
Reducibility occurs precisely when there is a unipotent $u\not=1$ in
$\widehat{G}$ with $sus^{-1}=u^q$, where $q$ is the residual cardinality
of $F$ (see e.g. [L]). 

In the quasisplit case, if $s=s'\sigma$, the relation becomes 
$s'\sigma(u)s'{}^{-1}=u^q$. If $G=\U(3,E/F)$, $E/F$ unramified
quadratic extension, thus the residual cardinality of $E$ is $q^2$, 
the representation $I(\eta)$ induced from the unramified character 
$\eta:t^n\mapsto a^n$, $t=\diag(\upi^{-1},1,\upi)$ ($\upi$ is a uniformizer
in $F^\times$) is parametrized by $s=s'\sigma$ with $s'=\diag(a,1,1)$ 
(which is in $\SL(3,\C)$ up to a scalar multiple; our representation has 
trivial central character so it can be viewed as one on the adjoint form 
of the group). Writing $u=[x,y,z]$ for the upper triangular unipotent 
matrix with top row $(1,x,y)$ and middle row $(0,1,z)$, we check that 
$\sigma(u)=[z,xz-y,x]$, $s'\sigma(u)s'{}^{-1}=[az,axz-ay,x]$, 
$u^q=[qx,qy+q(q-1)xz/2,qz]$. Suppose $s'\sigma(u)s'{}^{-1}=u^q$ and 
$u\not=1$. If $z\not=0$, then $x=qz$, $a=q^2$, $y=qz^2/2$. If $z=0$ 
then $x=0$ and $-ya=qy$ implies $a=-q$. Thus reducibility occurs in two cases:

\n (1) $a=q^2$, the constituents are the nontempered trivial representation
tr and the square integrable St;

\n (2) $a=-q$, the constituents are the nontempered representation which 
we denote by $\pi^\times$ and the square integrable $\pi^+$.

Put $r=\antidiag(1,-1,1)$ and $r'=rt$ for the
reflections in $G$ with $K=I\cup IrI$ and $K'=I\cup Ir'I$. The Iwahori 
algebra $H_I$ (of compactly supported $I$-biinvariant $\C$-valued
functions on $G$) is generated over $\C$ by the characteristic functions
$T$ of $IrI$ and $T'$ of $Ir'I$, subject to the relations $(T+1)(T-q^3)=0$
and $(T'+1)(T'-q)=0$; see, e.g., [Bo], 3.2(2). The characteristic functions 
of $K$ and $K'$ are $T_K=1+T$ and $T_{K'}=1+T'$. The functor $V\mapsto V^I$ 
is an equivalence from the category of representations of $G$ with a nonzero 
$I$-invariant vector to the category of $H_I$-modules. On the two dimensional
$H_I$-module $I(\eta)^I$ the element $TT'$ acts as $\delta^{1/2}(t)
\diag(\eta(\upi),\eta(\upi^{-1}))$ for some basis, where
$\delta(t)=|\det[\Ad(t)|\Lie N]|=q^4$, but $T$, $T'$ are not diagonalizable
with respect to a basis which diagonalizes $TT'$. When $I(\eta)$ is
reducible, the constituents correspond to one dimensional representations
of $H_I$. The possible images of $T$ are $-1$ and $q^3$, of $T'$ are
$-1$ and $q$. Thus on the trivial representation $(T,T')\mapsto(q^3,q)$,
and on the Steinberg $(T,T')\mapsto(-1,-1)$, so $TT'$ acts on the
corresponding induced $I(\eta)$ with eigenvalues $(q^4,1)=q^2(q^2,q^{-2})$,
and the induced is $I(\eta)$ with $\eta(t)=a$ equals $q^2$. On $\pi^\times$: 
$(T,T')\mapsto(q^3,-1)$, on $\pi^+$: $(T,T')\mapsto(-1,q)$, so $TT'$ has
eigenvalues $(-q^3,-q)=q^2(-q,-q^{-1})$ and the induced is $I(\eta)$ with
$\eta(t)=-q$. Now the eigenvalues of $(T_K,T_{K'})=(1+T,1+T')$ are on tr: 
$(1+q^3,1+q)$, on St: $(0,0)$, on $\pi^\times$: $(1+q^3,0)$, on $\pi^+$: 
$(0,1+q)$. We conclude that the trivial representation has both (nonzero) 
$K$ and $K'$-fixed vectors, the Steinberg has none, $\pi^\times$ has a 
$K$-fixed vector but no $K'$-fixed vector, and $\pi^+$ has a $K'$-fixed vector
but no $K$-fixed vector, thus $\pi^\times{}^K\not=0=\pi^\times{}^{K'}$ 
and $\pi^+{}^K=0\not=\pi^+{}^{K'}$. Clearly each $I(\eta)$ has both $K$ and
$K'$-fixed vectors.

\subsection{$\U(3)$ - the Split Case}
Let $E/F$ denote a totally imaginary quadratic extension $E$ of a totally
real number field $F$. Consider the quasi-split unitary $F$-group
$G^{\qs}=\U(2,1)$ in $3$ variables, split over $E$. Let $G=\U(3)$
be an arbitrary inner form of $G^{\qs}$ such that $G_{\infty}$ is
compact. Such exist since $E$ is CM. The rank is odd, so we
may even assume $G$ is quasi-split at all finite primes, but we do
not need that here. Let $\q$ be a prime of $F$ split in $E$. 
Denote by $R_\q$ the ring of integers in the completion $F_\q$ of $F$ 
at $\q$, and by $q=q_w$ the (residual) cardinality, of $\F_\q=R_\q/\q$.
Let $\q_E$ be a prime of $E$ over $\q$. Let us list the parahoric 
subgroups of $\GL(3,E_{\q_E}) \simeq \GL(3,F_\q)$. There is the 
hyperspecial maximal compact subgroup $K_\q=\GL(3,R_\q)$, 
and the Iwahori subgroup
$$
I_\q=\left\{g \in K_w: g \equiv \begin{pmatrix} * & * & * \\ 0 & * &
* \\ 0 & 0 & * \end{pmatrix}(\md \q)\right\}.
$$
There is only one $\GL(3,F_\q)$-conjugacy class of maximal proper
parahorics. Denote by $\upi_\q$ a generator of the maximal ideal 
$\q$ in the ring $R_\q$ of integers in $F_\q$. Put
$\mu_w=\diag(\upi_\q,\upi_\q,1)$. Then 
$$J_\q=\left\{g \in K_w: g \equiv \begin{pmatrix} * & * & * \\ *&* &
* \\ 0 & 0 & * \end{pmatrix}(\md\q)\right\}=K_\q \cap \mu_w^{-1}K_\q\mu_w. $$
is a representative. 

\subsection{Proof of Theorem 0.5} 
We first need to classify all the Iwahori-spherical
representations of $\GL(3,F_\q)$. It is a theorem of Borel [Bo] and
Bernstein [B] that these are precisely the constituents of the unramified
principal series. Let $\nu=|\cdot|$ be the absolute value character on 
$F_\q$. Using the theory of Bernstein and Zelevinsky [BZ] we obtain 
the following table.


\begin{tabular}{|c|c|c|c|c|c|}
\hline
   &  & constituent of & representation & unitary & tempered\\
 \hline \hline
  I &  & $\chi_1 \times \chi_2 \times \chi_3$ & $\chi_1 \times \chi_2 \times \chi_3$
  & below & $|\chi_i|=1$\\
\hline
 II & a & $\chi_1\nu^{1/2}\times\chi_1\nu^{-1/2}\times \chi_2$
  & $\chi_1\St_{\GL(2)}\times \chi_2$ & $|\chi_i|=1$  & $|\chi_i|=1$ \\
\cline{4-6}
   & b & $\chi_1\chi_2^{-1}\neq \nu^{\pm 3/2}$ & $\chi_1\one_{\GL(2)}\times \chi_2$
   & $|\chi_i|=1$  &  \\
 \hline
  III & a & $\chi\nu \times \chi \times \chi\nu^{-1}$ & $\chi\St_{\GL(3)}$
  & $|\chi|=1$  & $|\chi|=1$ \\
\cline{4-6}
  & b &  & $\chi V_P$ &  &\\
 \cline{4-6}
 & c &  & $\chi V_Q$ &  & \\
 \cline{4-6}
 & d &  & $\chi \one_{\GL(3)}$ & $|\chi|=1$ & \\
 \hline
\end{tabular}

{\center{Table A: Iwahori-spherical representations of GL(3)}\endcenter}

Only the representations of types I, IIa, IIIa are generic, and a representation
in Table A is square integrable iff it is of type IIIa and $|\chi|=1$.

The irreducible representation $\chi_1 \times \chi_2 \times\chi_2$ 
in case I is unitary if and only if either all the $\chi_i$ are unitary, or
$\chi_1\chi_2^{-1}=\nu^{\alpha}$ with $0<\alpha<1$ and $\chi_3$ 
unitary (after a permutation). In the table, $P$ and $Q$ denote the 
parabolics of $G=\GL(3,F_\q)$ of type $(2,1)$ and $(1,2)$ respectively. 
Moreover, $V_P=C^{\infty}(P\backslash G)/\C$ and 
$V_Q=C^{\infty}(Q\backslash G)/\C$. They are not
unitary, and therefore irrelevant for the theory of automorphic
forms. Next, we list the dimensions of their parahoric fixed spaces:

\begin{tabular}{|c|c|c|c|c|c|c|}
\hline
   &   & representation & remarks & $K$ & $J$ & $I$ \\
 \hline \hline
  I &   & $\chi_1 \times \chi_2 \times \chi_3$
  &  & 1 &  3 & 6  \\
\hline
  II & a
  & $\chi_1\St_{\GL(2)}\times \chi_2$ &   & 0 & 1 & 3  \\
\cline{3-7}
   & b & $\chi_1\one_{\GL(2)}\times \chi_2$
   &   &  1 &  2 & 3  \\
 \hline
  III & a & $\chi\St_{\GL(3)}$
  &   & 0 & 0 & 1  \\
\cline{3-7}
  & b & $\chi V_P$ & not unitary & 0 & 1 & 2 \\
 \cline{3-7}
 & c & $\chi V_Q$ & not unitary & 0 & 1 & 2 \\
 \cline{3-7}
 & d & $\chi\one_{\GL(3)}$  & irrelevant & 1 & 1 & 1 \\
 \hline
\end{tabular}

{\center{Table B: Dimensions of the parahoric fixed spaces}\endcenter}

To compute these dimensions, we use the following observation: If
$P$ is parabolic and $J$ is parahoric, a choice of representatives
$g \in P \backslash G / J$ determines an isomorphism
$$
\Ind_P^G(\tau)^J \simeq {\bigoplus}_{g \in P \backslash G / J}\,\,\,
\tau^{P \cap gJg^{-1}},
$$
for every representation $\tau$ of a Levi factor $M_P$. In
particular, if $P=B$ is the Borel subgroup and $\tau$ is an
unramified character, the dimension of $\Ind_B^G(\tau)^J$ equals
the number of double cosets $|B \backslash G / J|$. With this
information, the proof proceeds as follows. Theorem 0.3 gives
an automorphic representation $\wt{\pi}$ congruent to 
$\pi$ (modulo $\lambda$) such that $\wt{\pi}_\q^{J_\q} \neq
\wt{\pi}_\q^{K_\q}+\wt{\pi}_\q^{K_\q'}$. Since $\wt{\pi}_\q$
must be unitary, we see from table B that it is of type I or IIa.
Then, from table A, we derive that $\wt{\pi}_\q$ is generic and
not $L^2$. Finally, note that there is a bijection $K/J \simeq
\GL(3,\F_\q)/\bar{P}\simeq\mathbb P_2(\F_{q_w})$, whose cardinality
is $(q_w^3-1)/(q_w-1)$, so $[K:J]=1+q_w+q_w^2$, $q_w=|R_\q/\q|$,
but all maximal compact subgroups of $\GL(3,F_w)$ are conjugate,
so $[K'_w:J_w]_{K_w}=1$.
\hfill $\square$

\section{Applications for $\GSp(2)$}

In this section we view the symplectic group GSp(2) of rank two 
as an algebraic $F$-subgroup of GL(4) by realizing it with respect to 
the standard skew-diagonal symplectic form. With this choice, the set 
of upper triangular matrices form a Borel subgroup $B=TU$. There 
are two maximal parabolic subgroups containing $B$. One is the 
Siegel parabolic
$$
P=M_P \ltimes N_P=\left\{\begin{pmatrix} g & \\ & \nu\cdot {}^{\tau}\!g^{-1}
\end{pmatrix}
\begin{pmatrix}1 & & r & s \\ & 1 & t & r \\ & & 1 & \\ & & & 1
\end{pmatrix}\right\},
$$
where ${}^{\tau}\!g$ denotes the skew-transpose. The other is the 
Heisenberg parabolic
$$
Q=M_Q \ltimes N_Q=\left\{\begin{pmatrix} \nu & &   \\ & g &   \\ & &
\nu^{-1}\cdot\det g\end{pmatrix}
\begin{pmatrix} 1& c& &  \\ &1 &  & \\ &  & 1&-c \\  & & & 1
\end{pmatrix}
\begin{pmatrix} 1& & r& s \\ &1 &  & r\\ &  & 1& \\  & & & 1
\end{pmatrix}  \right\}.
$$

We consider an inner form $G$ of GSp(2) such that
$G^{\der}(\R)$ is compact. Concretely we have 
$G=\GSpin(f)$, where $f$ is some definite quadratic form in $5$ 
variables over $F$. Let us first describe the parahoric subgroups of
$\GSp(2,F_\q)$. There is the hyperspecial maximal compact subgroup
$K_\q=\GSp(2,R_\q)$, and the Iwahori subgroup $I_\q$ consisting 
of elements in $K_\q$ with upper triangular reduction mod $\q$.
Similarly, $P$ and $Q$ define (non-conjugate) parahoric subgroups
$J_\q'$ and $J_\q$ called the Siegel parahoric and the Heisenberg
parahoric respectively. One can easily check that we have the
identity,
$$
\text{$J_\q'=K_\q \cap h_wK_\q h_w^{-1},\quad$ where 
$h_w=\begin{pmatrix} & I \\  \upi_\q I & \end{pmatrix}$, 
$I=\begin{pmatrix} 1&\\& 1\end{pmatrix}$.}
$$
However, $J_\q=K_\q \cap K_\q'$, where $K_\q'$ is the non-special
{\it paramodular} (see [Sch], p. 267) maximal compact subgroup 
containing $I_\q$. Since $P$ and $Q$ are not 
associated parabolics, the classification of the Iwahori-spherical
representations of $\GSp(2,F_\q)$ is more complicated than
for $\GL(3,F_\q)$. This is reproduced in Appendix 2 as Table C
and Table D from Table 1 and Table 3 of R. Schmidt [Sch]. We 
use the notation from Appendix 2. 

\subsection{Proof of Theorem 0.6} We apply Theorem 0.3 
to the Heisenberg parahoric $J_\q$. An easy computation 
shows that $[K'_\q:J_\q]=q_w$ and $[K_\q:J_\q]=(q_w^4-1)/(q_w-1)$,
hence $[K'_\q:J_\q]_{K_w}=q_w$. We get an automorphic representation 
$\wt{\pi}$, congruent to $\pi$ modulo $\lambda$, 
such that the component at $\q$ satisfies the identity:
$$\wt{\pi}_\q^{J_\q}\neq \wt{\pi}_\q^{K_\q}+\wt{\pi}_\q^{K_\q'}.$$
In particular, $\wt{\pi}_\q^{J_\q}\neq 0$. We must have that
$\wt{\pi}_\q^{K_\q} \cap \wt{\pi}_\q^{K_\q'}=0$, for otherwise
$\dim \wt{\pi}_\q=1$ and therefore $\wt{\pi}$ is
one-dimensional by the strong approximation theorem. However,
$\pi$ is assumed to be non-abelian modulo $\lambda$. Thus,
equivalently we have
$$
\dim \wt{\pi}_\q^{J_\q} > \dim \wt{\pi}_\q^{K_\q}+\dim
\wt{\pi}_\q^{K_\q'}.
$$
From table 3 of [Sch, p. 269] (copied as Table D in Appendix 2), 
we deduce that this inequality is satisfied precisely when
$\wt{\pi}_\q$ is of type I, IIa, IIIa, IVb, IVc, Va or VIa.
However, the representations of type IVb and IVc are not unitary
and can therefore be ruled out immediately. We are then left with
the possible types I, IIa, IIIa, Va and VIa. Then from table 1 of
[Sch, p. 264] (copied as Table C in Appendix 2), we read off that
$\wt{\pi}_\q$ is generic. Indeed all the representations of type
Xa are generic, for X arbitrary.

Let us show that the types Va and VIa can also be ruled out
if we assume $q^4 \neq 1$ (mod $\lambda$). Suppose first that
$\wt{\pi}_\q$ is of type Va, that is, the unique
subrepresentation of some $|\cdot|\xi_0 \times \xi_0 \rtimes
|\cdot|^{-1/2}\sigma$ where $\xi_0$ has order two, in the notations
of Sally-Tadic [ST].  By the main theorem, the
center of the Heisenberg-Hecke algebra $Z(\HH_{J_\q,\Z})$ acts
on $\wt{\pi}_\q^{J_\q}$ by a character
$\eta_{\wt{\pi}_\q^{J_\q}}$ satisfying the congruence
$$
\text{$\eta_{\wt{\pi}_\q^{J_\q}}(\phi) \equiv
\eta_{\pi_\q^{K_\q}}(e_{K_\q}\ast \phi) \equiv
\eta_{\one}(e_{K_\q}\ast \phi)$ (mod $\lambda$)},
$$
for all $\phi \in Z(\HH_{J_\q,\Z})$. We get immediately
that the analogous statement is also true for the center of the
Iwahori-Hecke algebra $Z(\HH_{I_\q,\Z})$. This, however,
acts by a character on the Iwahori-fixed vectors in the principal
series $|\cdot|\xi_0 \times \xi_0 \rtimes |\cdot|^{-1/2}\sigma$
(for it has an unramified Langlands quotient, so it is generated by
any nonzero $K_\q$-fixed vector). Hence $Z(\HH_{I_\q,\Z})$
acts on every constituent of this principal series by the same
character $\eta_{\wt{\pi}_\q^{I_\q}}$. In particular, the action
of the spherical Hecke algebra $\HH_{K_\q,\Z} \simeq
Z(\HH_{I_\q,\Z})$ on the $K_\q$-fixed vectors of the
unramified quotient (type Vd) is given by a character congruent to
$\eta_{\one}$. In terms of their Satake parameters we
therefore must have (modulo the action of the Weyl group):
$$\diag(q_w^{-1/2}\sigma(\upi_\q),\,q_w^{-1/2}\xi_0\sigma(\upi_\q),
\,q_w^{1/2}\xi_0\sigma(\upi_\q),\,q_w^{1/2}\sigma(\upi_\q))$$
$$\equiv\diag(q_w^{-3/2},\,q_w^{-1/2},\,q_w^{1/2},\,q_w^{3/2})
(\text{mod}\,\lambda).$$
Since $\xi_0(\upi_\q)=-1$ we conclude that $q_w\equiv -1$ or 
$q_w^2 \equiv -1$ modulo $\lambda$. 

Secondly, assume $\wt{\pi}_\q$ is of type
VIa, that is, the unique irreducible subrepresentation of some
$|\cdot| \times \one \rtimes |\cdot|^{-1/2}\sigma$. Then, by
the argument above, we conclude that the unramified quotient of
this principal series must be congruent to $\one$. That is, in
terms of their Satake parameters:
$$\diag(q_w^{-1/2}\sigma(\upi_\q),\,q_w^{-1/2}\sigma(\upi_\q),\,q_w^{1/2}
\sigma(\upi_\q),\,q_w^{1/2}\sigma(\upi_\q))$$ 
$$\equiv\diag(q_w^{-3/2},\,q_w^{-1/2},\,q_w^{1/2},\,q_w^{3/2})(\text{mod}\,\lambda).$$
It follows that $q_w^2 \equiv 1$. The types I, IIa and IIIa cannot
be excluded, even if $\pi$ has trivial central character.\hfill $\square$

\section*{Appendix 1. Congruent Representations}

The compact open subgroups $K \subset G(\A^{\infty})$ form a
directed set by opposite inclusion, that is $K \preccurlyeq J
\Leftrightarrow K \supset J$. Let $R$ be a commutative ring. As
$K$ varies over the compact open subgroups, the centers
$Z(\HH_{K,R})$ form an inverse system of $R$-algebras with
respect to the canonical maps $Z(\HH_{K,R}) \leftarrow
Z(\HH_{J,R})$ when $K \supset J$. Let
$$
\ZZ_{G(\A^{\infty}),R}=\underleftarrow \lim
Z(\HH_{K,R}).
$$
In this limit, it is enough to let $K$ run through a neighborhood
basis at the identity. Thus $\ZZ_{G(\A^{\infty}),R}$ is a
commutative $R$-algebra, and it comes with projections ($K \supset J$)

\begin{align*}
\xymatrix{  & \ZZ_{G(\A^{\infty}),R} \ar[dl]_{\pr_K}
\ar[dr]^{\pr_J}
& \\
 Z(\HH_{K,R}) & &  Z(\HH_{J,R})
 \ar[ll]^{e_K \ast \phi \leftarrow \phi}}
\end{align*}

All we have said makes sense for any locally profinite group, so
in particular we have local analogues $\ZZ_{G_v,R}$ for
each finite place $v$. If $\mu=\otimes \mu_v$, it follows that
$$
\ZZ_{G(\A^{\infty}),R} \simeq 
\bigotimes_{v <\infty}\ZZ_{G_v,R},
$$
a restricted tensor product. Indeed the decomposable groups
$K=\prod K_v$ form a cofinal system. It remains to determine the
algebras $\ZZ_{G_v,R}$. By [B, 2.1], there exists a
neighborhood basis at $1$ consisting of compact open subgroups
$K_v \subset G_v$ with Iwahori factorization with respect to a
fixed minimal parabolic. If $G_v$ is unramified, for such a $K_v$
the canonical map $Z(\HH_{K_v,R})\to\HH_{v,R}^{\sph}$ to the 
spherical Hecke algebra at $v$ is an isomorphism [Bu]. This is a 
well-known result due to Bernstein when $K_v$ is an actual Iwahori 
subgroup. Therefore,
$$
\text{$G_v$ unramified $\Longrightarrow$ $\ZZ_{G_v,R}
\simeq \HH_{v,R}^{\sph}$.}
$$
The reason for introducing these objects is the following: Let
$\pi=\otimes \pi_v$ be an irreducible admissible representation of
$G(\A)$. Then there exists a unique character
$$
\eta_{\pi}:\ZZ_{G(\A^{\infty}),\Z} \to \C,
$$
such that $\eta_{\pi}=\eta_{\pi^K}\circ \pr_K$ for every $K$
such that $\pi^K \neq 0$. Uniqueness is clear, and the existence
reduces to showing that $\eta_{\pi^J}(\phi)=\eta_{\pi^K}(e_K \ast
\phi)$ for $K \supset J$ when $\pi^K \neq 0$. Similarly, we have
characters $\eta_{\pi_v}$ locally, and $\eta_{\pi}=\otimes
\eta_{\pi_v}$ under the isomorphism above. If $\pi$ is automorphic
and $\pi_{\infty}=\rho_\infty$, the character $\eta_{\pi}$ maps into
the ring of integers of some number field. Our work suggests the
following definition:

\begin{dfn}
Let $\pi$ and $\wt{\pi}$ be automorphic representations of
$G(\A)$, both $\rho_\infty$ at infinity. Let $\lambda$ be a finite
place of $\bar{\Q}$. We say that $\pi$ and $\wt{\pi}$ are
{\it congruent modulo} $\lambda$, and we write $\wt{\pi}\equiv \pi$
$(\md\lambda)$, if for all $\phi \in
\ZZ_{G(\A^{\infty}),\Z}$ we have
$$\eta_{\wt{\pi}}(\phi)\equiv \eta_{\pi}(\phi)\,\, (\md\lambda).$$
\end{dfn}

Analogously, it makes sense to say that the local components
$\wt{\pi}_v$ and $\pi_v$ are congruent. Then $\wt{\pi}\equiv
\pi$ (mod $\lambda$) if and only $\wt{\pi}_v\equiv \pi_v$ (mod
$\lambda$) for all $v < \infty$. 
Note also that if $\wt{\pi}_v$ and
$\pi_v$ are both unramified, then $\wt{\pi}_v\equiv \pi_v$ (mod
$\lambda$) means that the Satake parameters are congruent as it
should. With these definitions, our results translate into those
stated in the introduction.

\section*{Appendix 2. Iwahori-Spherical Representations of $\GSp(4)$}

In this appendix we reproduce parts of Table 1 and Table 3 in
[Sch]. The tables in [Sch] contain more information
than what is listed here (such as Atkin-Lehner eigenvalues and
signs of $\varepsilon$-factors). Below, we employ the notation of
[ST]. Thus $\nu$ denotes the normalized absolute value of a
non-archimedean local field. If $\chi_1$, $\chi_2$ and $\sigma$
are unramified characters, $\chi_1 \times \chi_2\rtimes \sigma$ 
denotes the principal series of $\GSp(2)$ obtained from
$$
T \ni \diag(x,y,zy^{-1},zx^{-1}) \mapsto
\chi_1(x)\chi_2(y)\sigma(z) \in \C^\times
$$
by normalized induction. Similarly, if $\pi$ is a representation
of $\GL(2)$, we denote by $\pi \rtimes \sigma$ and $\sigma \rtimes
\pi$ the representations of $\GSp(2)$ induced from
$\diag(X,z\cdot{}^{\tau}\!X^{-1}) \mapsto \pi(X)\sigma(z)$ and
$\diag(z,X,z^{-1}\det X) \mapsto \sigma(z)\pi(X)$ respectively. By
$L((-))$ we mean the unique irreducible quotient (the Langlands
quotient) when it exists. The representations
$\tau(S,\nu^{-1/2}\sigma)$ and $\tau(T,\nu^{-1/2}\sigma)$ are the
constituents of $\one\rtimes \sigma \St_{\GL(2)}$. 
They can be called limits of discrete series. The
nontrivial unramified quadratic character is denoted by $\xi_0$.

In the following Table C, a representation is generic iff it is of type
I or Xa, and $L^2$ iff it is of type IVa or Va.

In table D below, our notation is different from [Sch]: $K$ is hyperspecial, 
$K'$ is paramodular, $J$ is the Heisenberg parahoric, $J'$ the Siegel 
parahoric and $I$ is the Iwahori subgroup of GSp(4).

\begin{tabular}{|c|c|c|c|c|c|c|}
\hline
   &  & constituent of & representation &  tempered \\
 \hline \hline
  I &  & $\chi_1 \times \chi_2 \rtimes \sigma$
  & $\chi_1 \times \chi_2 \rtimes \sigma$  & $|\chi_i|=|\sigma|=1$  \\
\hline
  II & a & $\nu^{1/2}\chi \times \nu^{-1/2}\chi \rtimes \sigma$,
& $\chi\St_{\GL(2)}\rtimes \sigma$  &  $|\chi|=|\sigma|=1$ \\
\cline{4-5}
   & b & $\chi^2 \notin \{\nu^{\pm 1}, \nu^{\pm 3}\}$ & $\chi\one_{\GL(2)}\rtimes \sigma$
     &   \\
 \hline
  III & a & $\chi\times \nu \rtimes \nu^{-1/2}\sigma$, & $\chi \rtimes \sigma \St_{\GL(2)}$
   &  $|\chi|=|\sigma|=1$ \\
\cline{4-5}
  & b & $\chi \notin \{\one,\nu^{\pm 2}\}$ & $\chi \rtimes \sigma \one_{\GL(2)}$  &  \\
 \hline
 IV & a & $\nu^2 \times \nu \rtimes \nu^{-3/2}\sigma$ & $\sigma\St_{\GSp(4)}$  & $\bullet$  \\
 \cline{4-5}
 & b &  & $L((\nu^2,\nu^{-1}\sigma\St_{\GL(2)}))$   &   \\
\cline{4-5}
 & c &  & $L((\nu^{3/2}\St_{\GL(2)},\nu^{-3/2}\sigma))$  &  \\
\cline{4-5}
 & d &  & $\sigma\one_{\GSp(4)}$  &   \\
\hline
V & a & $\nu\xi_0 \times \xi_0 \rtimes \nu^{-1/2}\sigma$, & $\delta([\xi_0,\nu\xi_0],\nu^{-1/2}\sigma)$  & $\bullet$ \\
 \cline{4-5}
 & b & $\xi_0^2=\one$, $\xi_0 \neq \one$ & $L((\nu^{1/2}\xi_0\St_{\GL(2)},\nu^{-1/2}\sigma))$  &   \\
\cline{4-5}
 & c &  & $L((\nu^{1/2}\xi_0\St_{\GL(2)},\xi_0\nu^{-1/2}\sigma))$  &   \\
\cline{4-5}
 & d &  & $L((\nu\xi_0,\xi_0 \rtimes \nu^{-1/2}\sigma))$  &   \\
\hline
VI & a & $\nu \times \one \rtimes \nu^{-1/2}\sigma$ & $\tau(S,\nu^{-1/2}\sigma)$  & $\bullet$ \\
 \cline{4-5}
 & b &  & $\tau(T,\nu^{-1/2}\sigma)$  & $\bullet$ \\
\cline{4-5}
 & c &  & $L((\nu^{1/2}\St_{\GL(2)},\nu^{-1/2}\sigma))$  & \\
\cline{4-5}
 & d &  & $L((\nu, \one \rtimes \nu^{-1/2}\sigma ))$  &   \\
\hline

\end{tabular}

{\center{Table C: Iwahori-spherical representations of GSp(4)}\endcenter}

\begin{tabular}{|c|c|c|c|c|c|c|c|c|}
\hline
   &   & representation & remarks & $K$ & $K'$ & $J$ & $J'$ & $I$ \\
 \hline \hline
  I &   &
   $\chi_1 \times \chi_2 \rtimes \sigma$&  & 1 &  2 & 4 & 4 & 8\\
\hline
  II & a
  & $\chi\St_{\GL(2)}\rtimes \sigma$ &   & 0 & 1 & 2 & 1& 4\\
\cline{3-9}
   & b & $\chi\one_{\GL(2)}\rtimes \sigma$
   &   & 1  &  1 & 2 & 3& 4\\
 \hline
  III & a & $\chi \rtimes \sigma \St_{\GL(2)}$
  &   &  0 & 0 & 1 & 2& 4\\
\cline{3-9}
  & b & $\chi \rtimes \sigma \one_{\GL(2)}$ &  &  1& 2 & 3 & 2& 4\\
 \hline
IV & a & $\sigma\St_{\GSp(4)}$ &  & 0 &  0& 0 &0 & 1\\
 \cline{3-9}
 & b &  $L((\nu^2,\nu^{-1}\sigma\St_{\GL(2)}))$ & not unitary & 0 &  0& 1 & 2& 3\\
 \cline{3-9}
 & c & $L((\nu^{3/2}\St_{\GL(2)},\nu^{-3/2}\sigma))$  & not unitary & 0 & 1 &  2&1 &3 \\
 \cline{3-9}
 & d & $\sigma\one_{\GSp(4)}$  & irrelevant & 1 & 1 & 1 & 1& 1\\
\hline
V & a & $\delta([\xi_0,\nu\xi_0],\nu^{-1/2}\sigma)$ &  & 0 & 0 & 1 &0 & 2\\
 \cline{3-9}
 & b & $L((\nu^{1/2}\xi_0\St_{\GL(2)},\nu^{-1/2}\sigma))$  &  & 0 & 1 & 1 &1 & 2 \\
 \cline{3-9}
 & c &  $L((\nu^{1/2}\xi_0\St_{\GL(2)},\xi_0\nu^{-1/2}\sigma))$ &  & 0 & 1 & 1 &1 & 2\\
 \cline{3-9}
 & d & $L((\nu\xi_0,\xi_0 \rtimes \nu^{-1/2}\sigma))$  &  & 1 & 0 & 1 & 2& 2\\
\hline
VI & a & $\tau(S,\nu^{-1/2}\sigma)$ &  & 0 & 0 &1  & 1& 3\\
 \cline{3-9}
 & b & $\tau(T,\nu^{-1/2}\sigma)$  &  & 0 & 0 & 0 & 1& 1\\
 \cline{3-9}
 & c & $L((\nu^{1/2}\St_{\GL(2)},\nu^{-1/2}\sigma))$  &  & 0 & 1 & 1 & 0& 1\\
 \cline{3-9}
 & d & $L((\nu, \one \rtimes \nu^{-1/2}\sigma ))$  &  & 1 & 1 & 2 &2 & 3\\
\hline

\end{tabular}

{\center{Table D: Dimensions of the parahoric fixed spaces}\endcenter}

\end{document}